\documentclass[12pt,a4paper,twoside]{article}
\usepackage[colorlinks=true,linkcolor=blue,citecolor=blue]{hyperref}
\usepackage{amsmath,mathrsfs,multicol,amsfonts,mathtools,amsthm, bm,fancyhdr,lastpage,xcolor,titlesec,cleveref}
\usepackage[english]{babel}
\usepackage{nicematrix}
\usepackage{amssymb}
\usepackage[bottom,multiple]{footmisc}
\usepackage[top=2.5cm,bottom=2.5cm,left=2.5cm,right=2.5cm]{geometry}
\definecolor{myeditcolor}{named}{blue}
\titleformat{\section}[block]{\bfseries\large}{\thesection. }{2pt}{}
\theoremstyle{definition}
\newtheorem{theorem}{Theorem}[section]
\newtheorem{definition}[theorem]{Definition}

\newtheorem{remark}[theorem]{Remark}
\newtheorem{example}[theorem]{Example}

\begin{document}
\thispagestyle{empty}
\begin{center}
\noindent\textbf{{\Large Pseudo-Riemannian Algebraic Ricci Solitons on Four-Dimensional Lie Groups}}
\end{center}
\begin{center} \noindent Youssef Ayad\footnote{corresponding author: youssef.ayad@edu.umi.ac.ma}\\
\small{\textit{$^1$Faculty of sciences, Moulay Ismail University of Mekn\`{e}s}}\\
\small{\textit{B.P. 11201, Zitoune, Mekn\`{e}s}, Morocco}
\end{center}
\begin{abstract}
We investigate the conditions under which pseudo-Riemannian inner products induce pseudo-Riemannian algebraic Ricci solitons on four-dimensional Lie algebras. By analyzing the derivation condition $\operatorname{Ric} - \eta \operatorname{Id} \in \operatorname{Der}(\mathfrak{g})$ across the classification of four-dimensional real Lie algebras, we determine the explicit matrix forms that the Ricci operator of a pseudo-Riemannian algebraic Ricci soliton must take. We present two applications of our formalism on a chosen four-dimensional Lie algebra by exhibiting a pseudo-Riemannian algebraic Ricci soliton and a flat pseudo-Riemannian inner product, which is a trivial algebraic Ricci soliton.
\end{abstract}
\begin{center}
\end{center}
\begin{center}
\end{center}
\section{Introduction}

The study of Ricci solitons plays a central role in modern differential geometry, serving as a natural generalization of Einstein metrics and as a solution to the Ricci flow equation \cite{batat2017algebraic,chow}. Moreover, a pseudo-Riemannian metric $g$ on a smooth manifold $M$ is called a \textit{Ricci soliton} if the following equation holds 
$$\operatorname{Ric} + \mathcal{L}_X g = \lambda g,$$
where $\mathrm{Ric}$ is the Ricci tensor of $g$, $\mathcal{L}_X g$ denotes the Lie derivative along a vector field $X$, and $\lambda \in \mathbb{R}$. Ricci solitons provide a natural framework for understanding the formation of singularities in the Ricci flow and play a pivotal role in the classification of pseudo-Riemannian manifolds with special curvature properties.  

While classical studies often focus on \textit{Riemannian manifolds}, where the metric is positive-definite, the theory extends naturally to \textit{pseudo-Riemannian manifolds}, which allow metrics of indefinite signature. This generalization encompasses Lorentzian geometry, which is fundamental in general relativity, and introduces richer curvature phenomena, leading to new challenges and opportunities in geometric analysis. Jablonski has made substantial contributions to the development of the theory of Ricci solitons in \cite{jablonski2011concerning,jablonski2014homogeneous,jablonski2015homogeneous}. Einstein Lorentzian metrics and Lorentzian Ricci solitons on nilpotent Lie groups are studied in \cite{boucetta2020einstein,wears2017lorentzian,bakhshandeh2022lorentz}. However,  \cite{calvarusolorentzian} describes Lorentzian four-dimensional Lie groups, including a complete classification of the Einstein and Ricci parallel examples. In contrast, \cite{calvarusopseudo} investigates pseudo-Riemannian Lie groups, including a procedure to classify such pseudo-Riemannian Lie groups and a classification of the Einstein examples.

A particularly tractable and fruitful class of Ricci solitons arises in the context of \textit{Lie groups endowed with left-invariant pseudo-Riemannian metrics}. Moreover, for a Lie group $G$ with Lie algebra $\mathfrak{g}$, the Ricci soliton is called an \emph{algebraic Ricci soliton}, if its Ricci operator satisfies 
$$\operatorname{Ric} = \eta \, \mathrm{Id} + D,$$
where $\mathrm{Ric}$ is viewed as a linear operator on $\mathfrak{g}$, $\eta \in \mathbb{R}$, and $D$ is a derivation of the Lie algebra $\mathfrak{g}$. These \textit{algebraic Ricci solitons} not only generalize Einstein metrics but also provide a bridge between differential geometry and the algebraic structure of Lie groups, allowing the classification and construction of solitons via linear algebraic methods rather than through the direct analysis of nonlinear partial differential equations.  

Pseudo-Riemannian algebraic Ricci solitons thus occupy a unique intersection of geometry, algebra, and mathematical physics. They extend classical Riemannian results to indefinite metrics, offer insight into the geometric flow of left-invariant structures, and appear naturally in the study of self-similar Lorentzian spacetimes. Consequently, their systematic study enhances our understanding of both the geometric and algebraic properties of manifolds with indefinite metrics, while providing potential applications in theoretical physics. In \cite{calvarusofino}, the authors completely classify four-dimensional homogeneous pseudo-Riemannian manifolds with non-trivial isotropy that admit non-trivial homogeneous Ricci solitons, and exhibit new examples that are neither solvmanifolds nor symmetric.

The concept of an algebraic Ricci soliton was first introduced by Lauret in the Riemannian setting \cite{lauret2001ricci} and later extended to the pseudo-Riemannian case in \cite{batat2017algebraic}. Building on Lauret's work, Jablonski in \cite{jablonski2014homogeneous} generalized many of his results to solvmanifolds, i.e., Riemannian manifolds that admit a transitive solvable group of isometries. As a consequence, he established that all homogeneous Riemannian Ricci solitons are algebraic. Note that in the general pseudo-Riemannian context, there exist Ricci solitons on Lie groups that are not algebraic \cite{wears2017lorentzian}. The classification of Lorentzian algebraic Ricci solitons on three-dimensional Lie groups is given in \cite{batat2017algebraic}. The problem of classifying Lorentzian algebraic Ricci solitons on four-dimensional nilpotent Lie groups is addressed in \cite{aitbenhaddou2025lorentzian}. The latter work builds on \cite{bokan2015lorentz}, in which the authors classified Lorentzian metrics on four-dimensional nilpotent Lie groups.

In this paper, we investigate the conditions under which pseudo-Riemannian inner products induce pseudo-Riemannian algebraic Ricci solitons on four-dimensional Lie algebras. By analyzing the derivation condition $\operatorname{Ric} - \eta \operatorname{Id} \in \operatorname{Der}(\mathfrak{g})$ across the classification of four-dimensional real Lie algebras, we determine the explicit matrix forms that the Ricci operator of a pseudo-Riemannian algebraic Ricci soliton must take. We present two applications of our formalism on a chosen four-dimensional Lie algebra by exhibiting a pseudo-Riemannian algebraic Ricci soliton and a flat pseudo-Riemannian inner product, which is a trivial algebraic Ricci soliton.

\section{Preliminaries}
We consider $G$ to be a connected Lie group with Lie algebra $\mathfrak{g}$ of dimension $n$. The identity element of $G$ will be denoted by $e$.
\begin{definition}
A \emph{pseudo-Riemannian metric} $g$ on $G$ is a mapping that smoothly assigns to each point $p \in G$, a nondegenerate symmetric bilinear form $g_p$ of signature $(r, s)$ where $n = r + s$ on the tangent space $T_pG$; that is, in some basis of $T_pG$, the matrix of $g_p$ has exactly $r$ positive and $s$ negative eigenvalues.
\end{definition}
One can find a basis $\left\lbrace e_1, \ldots, e_n\right\rbrace$ of $T_pG$ such that the matrix of $g_p$ is given in this basis by:
$$\left[ g_p(e_i, e_j)\right] = J = \operatorname{diag}\left\lbrace \underbrace{1, \ldots, 1}_{\text{$r$ times}}, \underbrace{-1, \ldots, -1}_{\text{$s$ times}}\right\rbrace.$$
\begin{definition}
A \emph{pseudo-Riemannian inner product} on the Lie algebra $\mathfrak{g}$ is a nondegenerate, symmetric bilinear form of signature $(r,s)$.
\end{definition}
Let $p \in G$, the left translation by $p^{-1}$ on $G$ is the diffeomorphism of $G$ given by:
$$\ell_{p^{-1}} : G \longrightarrow G, \quad x \longmapsto p^{-1}x.$$
The differential of $\ell_{p^{-1}}$ at $p$ is the mapping $d_p \ell_{p^{-1}}: T_pG \longrightarrow T_eG \simeq \mathfrak{g}$.
\begin{definition}
A \emph{pseudo-Riemannian metric} $g$ on $G$ is said to be left invariant if the left translations $\ell_{p}$, $p \in G$ are isometries with respect to $g$. This means that
$$g_x\left(u, v\right) = g_{px}\left(d_x \ell_p(u), d_x \ell_p(v)\right) \quad \forall x \in G, \quad \forall u, v \in T_xG.$$
\end{definition}
\begin{remark}
If $g$ is a pseudo-Riemannian metric on $G$, then its value at $e$, i.e. $g_e$ is a pseudo-Riemannian inner product on $\mathfrak{g}$. Conversely, if $\langle \cdot, \cdot\rangle$ is a pseudo-Riemannian inner product on $\mathfrak{g}$, put
$$g_p(u, v) = \langle d_p \ell_{p^{-1}}(u), d_p \ell_{p^{-1}}(v)\rangle \quad \forall u, v \in T_pG.$$
Then $g$ is a left invariant pseudo-Riemannian metric on $G$.
\end{remark}
The following is a very important fact:
\begin{remark}
The geometric structures associated with the pseudo-Riemannian Lie group $(G, g)$, such as the Levi-Civita connection and the Ricci operator, can be recovered at the level of the pseudo-Riemannian Lie algebra $(\mathfrak{g}, \langle \cdot, \cdot \rangle)$. Henceforth, we will work exclusively with the pseudo-Riemannian Lie algebra $(\mathfrak{g}, \langle \cdot, \cdot \rangle)$.
\end{remark}
\subsection{The Ricci operator of a pseudo-Riemannian Lie algebra}
Let $(G, g)$ be a pseudo-Riemannian Lie group and let $\left(\mathfrak{g}, \langle \cdot, \cdot\rangle\right)$ be its associated pseudo-Riemannian Lie algebra. Consider $\nabla$ to be the Levi-Civita connection associated to $\left(\mathfrak{g}, \langle \cdot, \cdot\rangle\right)$; it is well known that $\nabla$ is characterized by the following Koszul formula
$$2\langle \nabla_u v, w\rangle = \langle [u, v], w\rangle + \langle [w, u], v\rangle + \langle [w, v], u\rangle, \quad \forall u, v, w \in \mathfrak{g}.$$
The Riemann curvature tensor $R$ of $\left(\mathfrak{g}, \langle \cdot, \cdot\rangle\right)$ associates to each pair $u, v \in \mathfrak{g}$ the linear transformation
$$R_{uv} = \nabla_{[u, v]} - [\nabla_u, \nabla_v].$$
\begin{definition}
\begin{enumerate}
\item A pseudo-Riemannian Lie algebra is called \emph{flat} if its Riemann curvature tensor is identically zero.
\item The \emph{Ricci curvature} $\operatorname{ric}$ is the symmetric tensor that associates to each pair of vectors $u, v \in \mathfrak{g}$, the trace of the endomorphism $w \longmapsto R_{uw}v$.
\item The \emph{Ricci operator} $\operatorname{Ric}$ is the endomorphism of $\mathfrak{g}$ given by the relation $$\langle \operatorname{Ric}(u), v\rangle = \operatorname{ric}(u, v).$$
\end{enumerate}
\end{definition}
\begin{remark}
Since the Ricci curvature $\operatorname{ric}$ is symmetric, we see that
$$\langle \operatorname{Ric}(u), v\rangle = \operatorname{ric}(u, v) = \operatorname{ric}(v, u) = \langle \operatorname{Ric}(v), u\rangle = \langle u, \operatorname{Ric}(v)\rangle.$$
Hence, the Ricci operator $\operatorname{Ric}$ is self-adjoint with respect to the inner product $\langle \cdot, \cdot\rangle$. Let $\mathrm{B} = \left\lbrace e_1, \dots, e_n \right\rbrace$ be a basis of $\mathfrak{g}$, and denote by $R$ and $G$ the matrices of the Ricci operator and the pseudo-Riemannian inner product in the basis $\mathrm{B}$, respectively. Then the self-adjointness of $\operatorname{Ric}$ with respect to the inner product is equivalent to $R^T G = G R$.
\end{remark}
\subsection{Procedure to compute the Ricci operator of a\\ pseudo-Riemannian Lie algebra}
Let $\mathscr{B} = \left\lbrace v_1, \ldots, v_n\right\rbrace$ be an orthonormal basis of $\left(\mathfrak{g}, \langle \cdot, \cdot\rangle\right)$, that is, it satisfies 
$$\langle v_i, v_j \rangle = 0 \quad \forall i \neq j, \quad \langle v_i, v_i \rangle = 1 \quad \forall i = 1, \ldots, r, \quad \langle v_i, v_i \rangle = -1 \quad \forall i = r + 1, \ldots, n.$$
The structure constants $\xi_{ijk}$ of $\left(\mathfrak{g}, \langle ., .\rangle\right)$ are defined by $\xi_{ijk} = \langle [v_i, v_j], v_k\rangle$. By the Koszul formula, the Levi-Civita connection of $\left(G, g\right)$ is described by (see \cite{ha2023left} for the particular Lorentzian case):
\begin{equation} \label{Levi}
\nabla_{v_i}v_j = \displaystyle{\sum_{k = 1}^{n} \langle v_k, v_k\rangle \frac{1}{2}\left(\xi_{ijk} - \xi_{jki} + \xi_{kij}\right)v_k}.
\end{equation}
Note that knowing the structure constants and the Levi-Civita connection of $\left(\mathfrak{g}, \langle \cdot, \cdot\rangle\right)$ allows us to find the Ricci operator $\operatorname{Ric}$ of $\left(\mathfrak{g}, \langle \cdot, \cdot\rangle\right)$. Furthermore, the Ricci operator of $\left(\mathfrak{g}, \langle \cdot, \cdot\rangle\right)$ is given by:
\begin{equation} \label{Ricci}
\operatorname{Ric}(u) = \displaystyle{\sum_{i = 1}^{n} \langle v_i, v_i\rangle R_{v_iu}v_i}.
\end{equation}
For example, if the dimension of $\mathfrak{g}$ is $3$ and the inner product $\langle \cdot, \cdot\rangle$ is Lorentzian, then the matrix $(R_{ij})$ of $\operatorname{Ric}$ in the orthonormal basis $\left\lbrace v_1, v_2, v_3\right\rbrace$ takes the form (see \cite{ha2023left}):
$$\operatorname{Ric} = \begin{bmatrix}
R_{11} & R_{12} & R_{13}\\
R_{12} & R_{22} & R_{23}\\
-R_{13} & -R_{23} & R_{33}
\end{bmatrix}.$$
This follows from the self-adjointness property $R^TG = GR$, where
$$G = J_{2, 1} = \operatorname{diag}\left\lbrace 1, 1, -1\right\rbrace$$
represents the Lorentzian inner product $\langle \cdot, \cdot\rangle$ in the orthonormal basis $\left\lbrace v_1, v_2, v_3\right\rbrace$.\\
If the dimension of $\mathfrak{g}$ is $4$ and the inner product $\langle \cdot, \cdot\rangle$ is Lorentzian, then the matrix $(R_{ij})$ of $\operatorname{Ric}$ in the orthonormal basis $\left\lbrace v_1, v_2, v_3, v_4\right\rbrace$ takes the form:
$$\operatorname{Ric} = \begin{bmatrix}
R_{11} & R_{12} & R_{13} & R_{14}\\
\\
R_{12} & R_{22} & R_{23} & R_{24}\\
\\
R_{13} & R_{23} & R_{33} & R_{34}\\
\\
-R_{14} & -R_{24} & -R_{34} & R_{44}
\end{bmatrix}.$$
This follows again from the self-adjointness property $R^TG = GR$, where
$$G = J_{3, 1} = \operatorname{diag}\left\lbrace 1, 1, 1, -1\right\rbrace$$
represents the Lorentzian inner product $\langle \cdot, \cdot\rangle$ in the orthonormal basis $\left\lbrace v_1, v_2, v_3, v_4\right\rbrace$.\\
But, if the inner product $\langle \cdot, \cdot\rangle$ has signature $(2, 2)$, then the matrix $(R_{ij})$ of $\operatorname{Ric}$ in the orthonormal basis $\left\lbrace v_1, v_2, v_3, v_4\right\rbrace$ takes the form:
$$\operatorname{Ric} = \begin{bmatrix}
R_{11} & R_{12} & R_{13} & R_{14}\\
\\
R_{12} & R_{22} & R_{23} & R_{24}\\
\\
-R_{13} & -R_{23} & R_{33} & R_{34}\\
\\
-R_{14} & -R_{24} & R_{34} & R_{44}
\end{bmatrix}.$$
This follows similarly from the self-adjointness property $R^TG = GR$, where
$$G = J_{2, 2} = \operatorname{diag}\left\lbrace 1, 1, -1, -1\right\rbrace$$
represents the pseudo-Riemannian inner product $\langle \cdot, \cdot\rangle$ in the orthonormal basis $\left\lbrace v_1, v_2, v_3, v_4\right\rbrace$.

Let $\mathcal{B} = \left\lbrace x_1, \ldots, x_n\right\rbrace$ be another, not necessarily orthonormal, basis for $\mathfrak{g}$. Setting $x_j = \displaystyle{\sum_{i} a_{ij}v_i}$ for some $a_{ij}$, a change of basis for the linear operator gives
$$\left[ \operatorname{Ric}\right]_{\left\lbrace x_i\right\rbrace} = A^{-1} \left[ \operatorname{Ric}\right]_{\left\lbrace v_i\right\rbrace} A,$$
where $\left[ \operatorname{Ric}\right]_{\left\lbrace x_i\right\rbrace}$ (respectively $\left[ \operatorname{Ric}\right]_{\left\lbrace v_i\right\rbrace}$) is the matrix of the Ricci operator $\operatorname{Ric}$ relative to the basis $\mathcal{B}$ (respectively $\mathscr{B}$) and $A = (a_{ij})$ is the transition matrix $\left[ \operatorname{id}\right]_{\left\lbrace x_i\right\rbrace, \left\lbrace v_i\right\rbrace}$.\vspace{0.5cm}

\textbf{The following observation is of central importance:}
\begin{remark}
In an arbitrary basis $\mathcal{B}$ of $(\mathfrak{g},\langle \cdot, \cdot \rangle)$, the matrix of the Ricci operator can be considered as an arbitrary $n \times n$ matrix $R = (R_{ij})$. Any symmetry or block structure of the matrix appears only when $\mathcal{B}$ is orthonormal with respect to $\langle \cdot, \cdot \rangle$. This allows us to work with the entries $R_{ij}$ freely in computations.
\end{remark}
\begin{definition}
A \emph{derivation} of $\mathfrak{g}$ is a linear map $D: \mathfrak{g} \longrightarrow \mathfrak{g}$ such that for all $x, y \in \mathfrak{g}$, the Leibniz rule holds:
$$D[x, y] = [D(x), y] + [x, D(y)].$$
\end{definition}
\begin{definition}
A pseudo-Riemannian inner product on $\mathfrak{g}$ is called an \emph{algebraic Ricci soliton} if its Ricci operator $\operatorname{Ric}$ satisfies the following equality $\operatorname{Ric} = \eta \mathit{Id}_{\mathfrak{g}} + D$ where $\eta$ is a real number and $D$ is a derivation of $\mathfrak{g}$.
\end{definition}
\begin{remark}
By saying that a pseudo-Riemannian inner product on the Lie algebra $\mathfrak{g}$ is an algebraic Ricci soliton, we mean that the left-invariant pseudo-Riemannian metric on the associated Lie group $G$, induced by this inner product, is an algebraic Ricci soliton.
\end{remark}
In this paper, we will focus on the case where the dimension of $\mathfrak{g}$ is $4$. We take $\left(\mathfrak{g}, \langle \cdot, \cdot\rangle\right)$ to be a four-dimensional pseudo-Riemannian Lie algebra with basis $\mathcal{B} = \left\lbrace e_1, e_2, e_3, e_4\right\rbrace$. The Ricci operator of the pseudo-Riemannian Lie algebra $\left(\mathfrak{g}, \langle \cdot, \cdot\rangle\right)$ is an arbitrary $4 \times 4$ matrix of the following form
$$\operatorname{Ric} = \begin{bmatrix}
R_{11} & R_{12} & R_{13} & R_{14}\\
\\
R_{21} & R_{22} & R_{23} & R_{24}\\
\\
R_{31} & R_{32} & R_{33} & R_{34}\\
\\
R_{41} & R_{42} & R_{43} & R_{44}
\end{bmatrix}.$$
Assume that the pseudo-Riemannian inner product $\langle \cdot, \cdot\rangle$ is an algebraic Ricci soliton, then there exists a constant $\eta$ and a derivation $D$ of $\mathfrak{g}$ such that the Ricci operator $\operatorname{Ric}$ of $\left(\mathfrak{g}, \langle \cdot, \cdot\rangle\right)$ satisfies $\operatorname{Ric} = \eta I_4 + D$. Then the derivation $D$ is given in the basis $\mathcal{B}$ by
\begin{equation} \label{derivation}
D = \begin{bmatrix}
R_{11} - \eta & R_{12} & R_{13} & R_{14}\\
\\
R_{21} & R_{22} - \eta & R_{23} & R_{24}\\
\\
R_{31} & R_{32} & R_{33} - \eta & R_{34}\\
\\
R_{41} & R_{42} & R_{43} & R_{44} - \eta
\end{bmatrix}.
\end{equation}
According to the bracket structure of each four-dimensional Lie algebra, we will study the conditions under which pseudo-Riemannian inner products on four-dimensional Lie algebras can give rise to algebraic Ricci solitons. Below are two tables that list the decomposable and indecomposable four-dimensional Lie algebras (see \cite{biggs2016classification}, pp. 1027 and 1028).
\begin{table}[htbp]
\caption{Decomposable four-dimensional Lie algebras}
\label{tab:decom_lie_algebras}
\centering
\begin{tabular}{|p{4cm}|p{11cm}|} \hline
Lie algebra  & Nonzero commutators  \\ \hline
$ 4\mathfrak{g}_1 $ &  \\ \hline
$ \mathfrak{g}_{2.1} \oplus 2\mathfrak{g}_1 $ & $[e_1, e_2] = e_1$ \\ \hline
$ 2\mathfrak{g}_{2.1} $ &  $[e_1, e_2] = e_1, \quad [e_3, e_4] = e_3$ \\ \hline
$ \mathfrak{g}_{3.1} \oplus \mathfrak{g}_1 $ & $[e_2, e_3] = e_1$ \\ \hline
$ \mathfrak{g}_{3.2} \oplus \mathfrak{g}_1 $ & $[e_2, e_3] = e_1 - e_2, \quad [e_3, e_1] = e_1$ \\ \hline
$ \mathfrak{g}_{3.3} \oplus \mathfrak{g}_1 $ & $[e_2, e_3] = -e_2, \quad [e_3, e_1] = e_1$ \\ \hline
$ \mathfrak{g}_{3.4}^0 \oplus \mathfrak{g}_1 $ & $[e_2, e_3] = e_1, \quad [e_3, e_1] = -e_2$ \\ \hline
$ \mathfrak{g}_{3.4}^{\alpha} \oplus \mathfrak{g}_1, \alpha > 0, \alpha \neq 1 $ & $[e_2, e_3] = e_1 - \alpha e_2, \quad [e_3, e_1] = \alpha e_1 - e_2$ \\ \hline
$ \mathfrak{g}_{3.5}^0 \oplus \mathfrak{g}_1 $ & $[e_2, e_3] = e_1, \quad [e_3, e_1] = e_2$ \\ \hline
$ \mathfrak{g}_{3.5}^{\alpha} \oplus \mathfrak{g}_1, \alpha > 0 $ & $[e_2, e_3] = e_1 - \alpha e_2, \quad [e_3, e_1] = \alpha e_1 + e_2$ \\ \hline
$ \mathfrak{g}_{3.6} \oplus \mathfrak{g}_1 $ & $[e_2, e_3] = e_1, \quad [e_3, e_1] = e_2, \quad [e_1, e_2] = -e_3$ \\ \hline
$ \mathfrak{g}_{3.7} \oplus \mathfrak{g}_1 $ & $[e_2, e_3] = e_1, \quad [e_3, e_1] = e_2, \quad [e_1, e_2] = e_3$ \\ \hline
\end{tabular}
\end{table}

\begin{table}[htbp]
\caption{Indecomposable four-dimensional Lie algebras}
\label{tab:indecom_lie_algebras}
\centering	
\begin{tabular}{|p{4cm}|p{11cm}|} \hline
Lie algebra  & Nonzero commutators  \\ \hline
$ \mathfrak{g}_{4.1} $ & $ [e_2, e_4] = e_1, \quad [e_3, e_4] = e_2 $ \\ \hline
$ \mathfrak{g}_{4.2}^{\alpha}, \alpha \neq 0 $ & $[e_1, e_4] = \alpha e_1, \quad [e_2, e_4] = e_2, \quad [e_3, e_4] = e_2 + e_3$ \\ \hline
$ \mathfrak{g}_{4.3} $ &  $[e_1, e_4] = e_1, \quad [e_3, e_4] = e_2$ \\ \hline
$ \mathfrak{g}_{4.4} $ & $[e_1, e_4] = e_1, \quad [e_2, e_4] = e_1 + e_2, \quad [e_3, e_4] = e_2 + e_3$ \\ \hline
$ \mathfrak{g}_{4.5}^{\alpha, \beta} $ & $[e_1, e_4] = e_1, \quad [e_2, e_4] = \beta e_2, \quad [e_3, e_4] = \alpha e_3$ \\ \hline
$ \mathfrak{g}_{4.6}^{\alpha, \beta}, \alpha > 0, \beta \in \mathbb{R} $ & $[e_1, e_4] = \alpha e_1, \quad [e_2, e_4] = \beta e_2 - e_3, \quad [e_3, e_4] = e_2 + \beta e_3$ \\ \hline
$ \mathfrak{g}_{4.7} $ & $[e_1, e_4] = 2e_1, \quad [e_2, e_4] = e_2, \quad [e_3, e_4] = e_2 + e_3, \quad [e_2, e_3] = e_1$ \\ \hline
$ \mathfrak{g}_{4.8}^{-1} $ & $[e_2, e_3] = e_1, \quad [e_2, e_4] = e_2, \quad [e_3, e_4] = -e_3$ \\ \hline
$ \mathfrak{g}_{4.8}^{\alpha}, -1 < \alpha \leq 1 $ & $[e_1, e_4] = (1 + \alpha)e_1, \quad [e_2, e_4] = e_2, \quad [e_3, e_4] = \alpha e_3, \newline [e_2, e_3] = e_1$ \\ \hline
$ \mathfrak{g}_{4.9}^0 $ & $[e_2, e_3] = e_1, \quad [e_2, e_4] = -e_3, \quad [e_3, e_4] = e_2$ \\ \hline
$ \mathfrak{g}_{4.9}^{\alpha}, \alpha > 0 $ & $[e_1, e_4] = 2\alpha e_1, \quad [e_2, e_4] = \alpha e_2 - e_3, \newline [e_3, e_4] = e_2 + \alpha e_3, \quad [e_2, e_3] = e_1 $ \\ \hline
$ \mathfrak{g}_{4.10} $ & $[e_1, e_3] = e_1, \quad [e_2, e_3] = e_2, \quad [e_1, e_4] = -e_2, \quad [e_2, e_4] = e_1$ \\ \hline
\end{tabular}
\end{table}

In table \ref{tab:indecom_lie_algebras}, for the Lie algebra $\mathfrak{g}_{4.5}^{\alpha, \beta}$, the conditions over $\alpha$ and $\beta$ are: 
$$-1 < \alpha \leq \beta \leq 1, \quad \alpha\beta \neq 0\qquad \text{or}\qquad \alpha = -1, \quad 0 < \beta \leq 1.$$
\section{Pseudo-Riemannian algebraic Ricci solitons on decomposable four-dimensional Lie algebras}
\subsection{The Lie algebra $4\mathfrak{g}_1$}
Since $4\mathfrak{g}_1$ is abelian, every pseudo-Riemannian inner product on it is flat.
\subsection{The Lie algebra $\mathfrak{g}_{2.1} \oplus 2\mathfrak{g}_1$}
The Lie algebra $\mathfrak{g}_{2.1} \oplus 2\mathfrak{g}_1$ has a basis $\mathcal{B} = \left\lbrace e_1, e_2, e_3, e_4\right\rbrace$ such that the only nonzero Lie bracket is $[e_1, e_2] = e_1$.
\begin{theorem}
A pseudo-Riemannian inner product on $\mathfrak{g}_{2.1} \oplus 2\mathfrak{g}_1$ is an algebraic Ricci soliton if and only if its Ricci operator satisfies
$$R_{21} = R_{31} = R_{41} = R_{13} = R_{14} = R_{23} = R_{24} = 0.$$
In this situation, the Ricci operator can be written as $\operatorname{Ric} = \eta I_4 + D$ where $\eta = R_{22}$ and
$$D = \begin{bmatrix}
R_{11} - R_{22} & R_{12} & 0 & 0\\
\\
0 & 0 & 0 & 0\\
\\
0 & R_{32} & R_{33} - R_{22} & R_{34}\\
\\
0 & R_{42} & R_{43} & R_{44} - R_{22}
\end{bmatrix}$$
is a derivation of $\mathfrak{g}_{2.1} \oplus 2\mathfrak{g}_1$ with respect to the basis $\mathcal{B}$.
\end{theorem}
\begin{proof}
Since $D$ in (\ref{derivation}) is a derivation of $\mathfrak{g}_{2.1} \oplus 2\mathfrak{g}_1$, it satisfies the following condition
\begin{eqnarray*}
& [D(e_1), e_2] + [e_1, D(e_2)] = D[e_1, e_2] \hspace{6.2cm}\\ \Leftrightarrow & \left[ \left( R_{11} - \eta\right)e_1 + R_{21}e_2 + R_{31}e_3 + R_{41}e_4, e_2\right] + \hspace{4.5cm}\\ & \left[ e_1, R_{12}e_1 + \left(R_{22} - \eta\right)e_2 + R_{32}e_3 + R_{42}e_4\right] = D(e_1) \hspace{3.4cm}\\
\Leftrightarrow& \left( R_{11} - \eta\right)e_1 + \left( R_{22} - \eta\right)e_1 = \left( R_{11} - \eta\right)e_1 + R_{21}e_2 + R_{31}e_3 + R_{41}e_4\hspace{0.4cm}
\end{eqnarray*}
By identifying the components of $e_i$, we obtain that
$$\left\lbrace\begin{array}{lll}
R_{11} - \eta + R_{22} - \eta = R_{11} - \eta \Longrightarrow \eta = R_{22} \\
\\
R_{21} = R_{31} = R_{41} = 0 
\end{array}\right.$$
By using the result obtained above and by moving to the bracket $[e_1, e_3] = 0$, we obtain that
\begin{eqnarray*}
[D(e_1), e_3] + [e_1, D(e_3)] = 0 &\Leftrightarrow& [e_1, R_{13}e_1 + R_{23}e_2 + (R_{33} - \eta)e_3 + R_{43}e_4] = 0\\
&\Leftrightarrow& R_{23}e_1 = 0 \Leftrightarrow R_{23} = 0
\end{eqnarray*}
Similarly, we can easily see that
$$[D(e_1), e_4] + [e_1, D(e_4)] = 0 \Leftrightarrow R_{24} = 0$$
$$[D(e_2), e_3] + [e_2, D(e_3)] = 0 \Leftrightarrow R_{13} = 0$$
$$[D(e_2), e_4] + [e_2, D(e_4)] = 0 \Leftrightarrow R_{14} = 0$$
But the condition $[D(e_3), e_4] + [e_3, D(e_4)] = 0$ is trivial. Thus, we have obtained the conclusions of the theorem.
\end{proof}
\subsection{The Lie algebra $2\mathfrak{g}_{2.1}$}
The Lie algebra $2\mathfrak{g}_{2.1}$ has a basis $\mathcal{B} = \left\lbrace e_1, e_2, e_3, e_4\right\rbrace$ such that the nonzero Lie brackets are 
$$[e_1, e_2] = e_1, \qquad [e_3, e_4] = e_3$$
\begin{theorem}
A pseudo-Riemannian inner product on $2\mathfrak{g}_{2.1}$ is an algebraic Ricci soliton if and only if its Ricci operator fulfills the following relations
\begin{multicols}{2}
\begin{itemize}
\item $R_{i1} = 0, \; i = 2, 3, 4$
\item $R_{32} = R_{42} = R_{14} = R_{24} = 0$
\item $R_{i3} = 0, \; i = 1, 2, 4$
\item $R_{22} = R_{44}$
\end{itemize}
\end{multicols}
With these relations in place, the Ricci operator may be written as $\operatorname{Ric} = \eta I_4 + D$ where $\eta = R_{22}$ and
$$D = \begin{bmatrix}
R_{11} - R_{22} & R_{12} & 0 & 0\\
\\
0 & 0 & 0 & 0\\
\\
0 & 0 & R_{33} - R_{22} & R_{34}\\
\\
0 & 0 & 0 & 0
\end{bmatrix}.$$
is a derivation of $2\mathfrak{g}_{2.1}$ with respect to the basis $\mathcal{B}$.
\end{theorem}
\begin{proof}
Given that $D$ in (\ref{derivation}) is a derivation of $2\mathfrak{g}_{2.1}$, it satisfies the following condition
\begin{eqnarray*}
& [D(e_1), e_2] + [e_1, D(e_2)] = D[e_1, e_2] \hspace{6.3cm}\\ \Leftrightarrow & \left[ \left( R_{11} - \eta\right)e_1 + R_{21}e_2 + R_{31}e_3 + R_{41}e_4, e_2\right] + \hspace{4.5cm}\\ & \left[ e_1, R_{12}e_1 + \left(R_{22} - \eta\right)e_2 + R_{32}e_3 + R_{42}e_4\right] = D(e_1) \hspace{3.5cm}\\
\Leftrightarrow& \left( R_{11} - \eta\right)e_1 + \left( R_{22} - \eta\right)e_1 = \left( R_{11} - \eta\right)e_1 + R_{21}e_2 + R_{31}e_3 + R_{41}e_4\hspace{0.6cm}
\end{eqnarray*}
By examining the components of $e_i$, we find that
$$\left\lbrace\begin{array}{lll}
R_{11} - \eta + R_{22} - \eta = R_{11} - \eta \Longrightarrow \eta = R_{22} \\
\\
R_{21} = R_{31} = R_{41} = 0 
\end{array}\right.$$
Then the derivation $D$ becomes
$$D = \begin{bmatrix}
R_{11} - R_{22} & R_{12} & R_{13} & R_{14}\\
\\
0 & 0 & R_{23} & R_{24}\\
\\
0 & R_{32} & R_{33} - R_{22} & R_{34}\\
\\
0 & R_{42} & R_{43} & R_{44} - R_{22}
\end{bmatrix}.$$
Next, $D$ satisfies the following property
\begin{eqnarray*}
& [D(e_3), e_4] + [e_3, D(e_4)] = D[e_3, e_4] \hspace{6.9cm}\\ \Leftrightarrow & \left[ R_{13}e_1 + R_{23}e_2 + \left( R_{33} - R_{22}\right)e_3 + R_{43}e_4, e_4\right] + \hspace{4.8cm}\\ & \left[ e_3, R_{14}e_1 + R_{24}e_2 + R_{34}e_3 + \left(R_{44} - R_{22}\right)e_4\right] = D(e_3) \hspace{3.7cm}\\
\Leftrightarrow& \left( R_{33} - R_{22}\right)e_3 + \left( R_{44} - R_{22}\right)e_3 = R_{13}e_1 + R_{23}e_2 + \left( R_{33} - R_{22}\right)e_3 + R_{43}e_4
\end{eqnarray*}
Identifying the components of $e_i$, yields
$$\left\lbrace\begin{array}{lll}
R_{33} - R_{22} + R_{44} - R_{22} = R_{33} - R_{22} \Longrightarrow R_{22} = R_{44} \\
\\
R_{13} = R_{23} = R_{43} = 0 
\end{array}\right.$$
Then the derivation $D$ becomes
$$D = \begin{bmatrix}
R_{11} - R_{22} & R_{12} & 0 & R_{14}\\
\\
0 & 0 & 0 & R_{24}\\
\\
0 & R_{32} & R_{33} - R_{22} & R_{34}\\
\\
0 & R_{42} & 0 & 0
\end{bmatrix}.$$
Now, it is easy to check that the condition $[D(e_1), e_3] + [e_1, D(e_3)] = 0$ holds trivially, and we obtain that
\begin{eqnarray*}
\left[D(e_1), e_4\right] + \left[e_1, D(e_4)\right] & = & 0 \Leftrightarrow R_{24} = 0\\
\left[D(e_2), e_3\right] + \left[e_2, D(e_3)\right] & = & 0 \Leftrightarrow R_{42} = 0\\
\left[D(e_2), e_4\right] + \left[e_2, D(e_4)\right] & = & 0 \Leftrightarrow R_{32} = R_{14} = 0
\end{eqnarray*}
Hence, the results asserted by the theorem follow.
\end{proof}
\subsection{The Lie algebra $\mathfrak{g}_{3.1} \oplus \mathfrak{g}_1$}
The Lie algebra $\mathfrak{g}_{3.1} \oplus \mathfrak{g}_1$ has a basis $\mathcal{B} = \left\lbrace e_1, e_2, e_3, e_4\right\rbrace$ such that the only nonzero Lie bracket is $[e_2, e_3] = e_1$.
\begin{theorem} \label{theorem3.3}
A pseudo-Riemannian inner product on $\mathfrak{g}_{3.1} \oplus \mathfrak{g}_1$ is an algebraic Ricci soliton if and only if its Ricci operator satisfies the conditions below
$$R_{21} = R_{31} = R_{41} = R_{24} = R_{34} = 0 .$$
Under the stated conditions, the Ricci operator can be expressed as $\operatorname{Ric} = \eta I_4 + D$ where $\eta = R_{22} + R_{33} - R_{11}$ and
$$D = \begin{bmatrix}
R_{11} - \eta & R_{12} & R_{13} & R_{14}\\
\\
0 & R_{22} - \eta & R_{23} & 0\\
\\
0 & R_{32} & R_{33} - \eta & 0\\
\\
0 & R_{42} & R_{43} & R_{44} - \eta
\end{bmatrix}$$
is a derivation of $\mathfrak{g}_{3.1} \oplus \mathfrak{g}_1$ with respect to the basis $\mathcal{B}$.
\end{theorem}
\begin{proof}
As a derivation of $\mathfrak{g}_{3.1} \oplus \mathfrak{g}_1$, $D$ in (\ref{derivation}) satisfies the following condition
\begin{eqnarray*}
& [D(e_2), e_3] + [e_2, D(e_3)] = D[e_2, e_3] \hspace{6.2cm}\\ \Leftrightarrow & \left[ R_{12}e_1 + \left( R_{22} - \eta\right)e_2 + R_{32}e_3 + R_{42}e_4, e_3\right] + \hspace{4.5cm}\\ & \left[ e_2, R_{13}e_1 + R_{23}e_2 + \left(R_{33} - \eta\right)e_3 + R_{43}e_4\right] = D(e_1) \hspace{3.4cm}\\
\Leftrightarrow& \left( R_{22} - \eta\right)e_1 + \left( R_{33} - \eta\right)e_1 = \left( R_{11} - \eta\right)e_1 + R_{21}e_2 + R_{31}e_3 + R_{41}e_4\hspace{0.4cm}
\end{eqnarray*}
An analysis of the components of $e_i$ shows that
$$\left\lbrace\begin{array}{lll}
R_{22} - \eta + R_{33} - \eta = R_{11} - \eta \Longrightarrow \eta = R_{22} + R_{33} - R_{11} \\
\\
R_{21} = R_{31} = R_{41} = 0 
\end{array}\right.$$
According to this fact, the derivation $D$ becomes
$$D = \begin{bmatrix}
R_{11} - \eta & R_{12} & R_{13} & R_{14}\\
\\
0 & R_{22} - \eta & R_{23} & R_{24}\\
\\
0 & R_{32} & R_{33} - \eta & R_{34}\\
\\
0 & R_{42} & R_{43} & R_{44} - \eta
\end{bmatrix}.$$
Now, it is easy to see that the following three conditions are trivial:
$$[D(e_1), e_i] + [e_1, D(e_i)] = 0 \quad \text{for}\quad i = 2, 3, 4.$$
But we can see that
$$[D(e_2), e_4] + [e_2, D(e_4)] = 0 \Leftrightarrow R_{34} = 0$$
$$[D(e_3), e_4] + [e_3, D(e_4)] = 0 \Leftrightarrow R_{24} = 0$$
Consequently, we have obtained all the required implications of the theorem.
\end{proof}
\subsection{The Lie algebra $\mathfrak{g}_{3.2} \oplus \mathfrak{g}_1$}
The Lie algebra $\mathfrak{g}_{3.2} \oplus \mathfrak{g}_1$ has a basis $\mathcal{B} = \left\lbrace e_1, e_2, e_3, e_4\right\rbrace$ such that the nonzero Lie brackets are $$[e_2, e_3] = e_1 - e_2, \qquad [e_3, e_1] = e_1.$$
\begin{theorem}
A pseudo-Riemannian inner product on $\mathfrak{g}_{3.2} \oplus \mathfrak{g}_1$ is an algebraic Ricci soliton if and only if its associated Ricci operator satisfies the conditions
\begin{multicols}{2}
\begin{itemize}
\item $R_{i1} = 0, \; i = 2, 3, 4$
\item $R_{32} = R_{42} = 0$
\item $R_{i4} = 0, \; i = 1, 2, 3$
\item $R_{22} = R_{11}$
\end{itemize}
\end{multicols}
Provided these conditions are satisfied, the Ricci operator can take the form\\ $\operatorname{Ric} = \eta I_4 + D$ where $\eta = R_{33}$ and
$$D = \begin{bmatrix}
R_{11} - R_{33} & R_{12} & R_{13} & 0\\
\\
0 & R_{11} - R_{33} & R_{23} & 0\\
\\
0 & 0 & 0 & 0\\
\\
0 & 0 & R_{43} & R_{44} - R_{33}
\end{bmatrix}$$
is a derivation of $\mathfrak{g}_{3.2} \oplus \mathfrak{g}_1$ with respect to the basis $\mathcal{B}$.
\end{theorem}
\begin{proof}
As $D$ in (\ref{derivation}) serves as a derivation of $\mathfrak{g}_{3.2} \oplus \mathfrak{g}_1$, the following condition holds
\begin{eqnarray*}
& [D(e_2), e_3] + [e_2, D(e_3)] = D[e_2, e_3] \hspace{7.3cm}\\ \Leftrightarrow & \left[ R_{12}e_1 + \left( R_{22} - \eta\right)e_2 + R_{32}e_3 + R_{42}e_4, e_3\right] + \hspace{5.6cm}\\ & \left[ e_2, R_{13}e_1 + R_{23}e_2 + \left(R_{33} - \eta\right)e_3 + R_{43}e_4\right] = D(e_1) - D(e_2) \hspace{3cm}\\
\Leftrightarrow& -R_{12}e_1 + \left( R_{22} - \eta\right)(e_1 - e_2) + \left( R_{33} - \eta\right)(e_1 - e_2) \hspace{4.6cm}\\
& = \left( R_{11} - \eta - R_{12}\right)e_1 + \left(R_{21} - R_{22} + \eta\right)e_2 + \left(R_{31} - R_{32}\right)e_3 + \left( R_{41} - R_{42}\right)e_4
\end{eqnarray*}
Upon identifying the components of $e_i$, we conclude that
$$\left\lbrace\begin{array}{lll}
- R_{12} + R_{22} - \eta + R_{33} - \eta = R_{11} - \eta - R_{12} \Longrightarrow \eta = R_{22} + R_{33} - R_{11} \\
\\
\eta - R_{22} + \eta - R_{33} = R_{21} - R_{22} + \eta \Longrightarrow \eta = R_{33} + R_{21}\\
\\
R_{31} = R_{32}, \quad R_{41} = R_{42}
\end{array}\right.$$
Similarly, $D$ satisfies the following condition
\begin{eqnarray*}
& [D(e_3), e_1] + [e_3, D(e_1)] = D[e_3, e_1] \hspace{8.6cm}\\ \Leftrightarrow & \left[ R_{13}e_1 + R_{23}e_2 + \left( R_{33} - \eta\right)e_3 + R_{43}e_4, e_1\right] + \hspace{6.9cm}\\ & \left[ e_3, \left( R_{11} - \eta\right)e_1 + R_{21}e_2 + R_{31}e_3 + R_{41}e_4\right] = D(e_1) \hspace{5.8cm}\\
\Leftrightarrow& \left( R_{33} - \eta\right)e_1 + \left( R_{11} - \eta\right)e_1 + R_{21}(e_2 - e_1) = \left( R_{11} - \eta\right)e_1 + R_{21}e_2 + R_{31}e_3 + R_{41}e_4
\end{eqnarray*}
From the components of $e_i$, we obtain
$$\left\lbrace\begin{array}{lll}
R_{33} - \eta + R_{11} - \eta - R_{21} = R_{11} - \eta \Longrightarrow \eta = R_{33} - R_{21} \\
\\
R_{31} = R_{41} = 0 
\end{array}\right.$$
We have $\eta = R_{33} + R_{21} = R_{33} - R_{21}$, hence $R_{21} = 0$. Then $\eta = R_{33} = R_{22} + R_{33} - R_{11}$, thus $R_{22} = R_{11}$. Therefore, the derivation $D$ becomes
$$D = \begin{bmatrix}
R_{11} - R_{33} & R_{12} & R_{13} & R_{14}\\
\\
0 & R_{11} - R_{33} & R_{23} & R_{24}\\
\\
0 & 0 & 0 & R_{34}\\
\\
0 & 0 & R_{43} & R_{44} - R_{33}
\end{bmatrix}.$$
Now, it is easy to see that the condition $[D(e_1), e_2] + [e_1, D(e_2)] = 0$ is trivial. But the condition $[D(e_1), e_4] + [e_1, D(e_4)] = 0$ gives $R_{34} = 0$. Since $R_{34} = 0$, the condition $[D(e_2), e_4] + [e_2, D(e_4)] = 0$ is trivial. Again, the condition $[D(e_3), e_4] + [e_3, D(e_4)] = 0$ forces $R_{14} = R_{24} = 0$. Thus, the theorem's conclusions follow.
\end{proof}
\subsection{The Lie algebra $\mathfrak{g}_{3.3} \oplus \mathfrak{g}_1$}
The Lie algebra $\mathfrak{g}_{3.3} \oplus \mathfrak{g}_1$ has a basis $\mathcal{B} = \left\lbrace e_1, e_2, e_3, e_4\right\rbrace$ such that the nonzero Lie brackets are $$[e_2, e_3] = - e_2, \qquad [e_3, e_1] = e_1.$$
\begin{theorem}
A pseudo-Riemannian inner product on $\mathfrak{g}_{3.3} \oplus \mathfrak{g}_1$ is an algebraic Ricci soliton if and only if its Ricci operator satisfies
$$R_{31} = R_{41} = R_{32} = R_{42} = R_{14} = R_{24} = R_{34} = 0.$$
In this case, the Ricci operator can be stated as $\operatorname{Ric} = \eta I_4 + D$ where $\eta = R_{33}$ and
$$D = \begin{bmatrix}
R_{11} - R_{33} & R_{12} & R_{13} & 0\\
\\
R_{21} & R_{22} - R_{33} & R_{23} & 0\\
\\
0 & 0 & 0 & 0\\
\\
0 & 0 & R_{43} & R_{44} - R_{33}
\end{bmatrix}$$
is a derivation of $\mathfrak{g}_{3.3} \oplus \mathfrak{g}_1$ with respect to the basis $\mathcal{B}$.
\end{theorem}
\begin{proof}
Being a derivation of $\mathfrak{g}_{3.3} \oplus \mathfrak{g}_1$, $D$ in (\ref{derivation}) satisfies the following condition
\begin{eqnarray*}
& [D(e_2), e_3] + [e_2, D(e_3)] = D[e_2, e_3] \hspace{7.6cm}\\ \Leftrightarrow & \left[ R_{12}e_1 + \left( R_{22} - \eta\right)e_2 + R_{32}e_3 + R_{42}e_4, e_3\right] + \hspace{5.8cm}\\ & \left[ e_2, R_{13}e_1 + R_{23}e_2 + \left(R_{33} - \eta\right)e_3 + R_{43}e_4\right] = - D(e_2) \hspace{4.3cm}\\
\Leftrightarrow& -R_{12}e_1 - \left( R_{22} - \eta\right)e_2 - \left( R_{33} - \eta\right)e_2 = -R_{12}e_1 - \left(R_{22} - \eta\right)e_2 - R_{32}e_3 - R_{42}e_4
\end{eqnarray*}
By identifying each component of $e_i$, we deduce that
$$\left\lbrace\begin{array}{lll}
R_{22} - \eta + R_{33} - \eta = R_{22} - \eta \Longrightarrow \eta = R_{33} \\
\\
R_{32} = R_{42} = 0
\end{array}\right.$$
Then the derivation $D$ becomes
$$D = \begin{bmatrix}
R_{11} - R_{33} & R_{12} & R_{13} & R_{14}\\
\\
R_{21} & R_{22} - R_{33} & R_{23} & R_{24}\\
\\
R_{31} & 0 & 0 & R_{34}\\
\\
R_{41} & 0 & R_{43} & R_{44} - R_{33}
\end{bmatrix}.$$
Again, $D$ satisfies the following condition
\begin{eqnarray*}
& [D(e_3), e_1] + [e_3, D(e_1)] = D[e_3, e_1] \hspace{8.5cm}\\ \Leftrightarrow & \left[ R_{13}e_1 + R_{23}e_2 + R_{43}e_4, e_1\right] + \left[ e_3, \left( R_{11} - R_{33}\right)e_1 + R_{21}e_2 + R_{31}e_3 + R_{41}e_4\right] = D(e_1) \\
\Leftrightarrow& \left( R_{11} - R_{33}\right)e_1 + R_{21}e_2 = \left( R_{11} - R_{33}\right)e_1 + R_{21}e_2 + R_{31}e_3 + R_{41}e_4\hspace{3cm}
\end{eqnarray*}
This shows that $R_{31} = R_{41} = 0$. Then the derivation $D$ becomes
$$D = \begin{bmatrix}
R_{11} - R_{33} & R_{12} & R_{13} & R_{14}\\
\\
R_{21} & R_{22} - R_{33} & R_{23} & R_{24}\\
\\
0 & 0 & 0 & R_{34}\\
\\
0 & 0 & R_{43} & R_{44} - R_{33}
\end{bmatrix}.$$
Now, it is easy to see that the condition $[D(e_1), e_2] + [e_1, D(e_2)] = 0$ is trivial. But the condition $[D(e_1), e_4] + [e_1, D(e_4)] = 0$ gives $R_{34} = 0$. Since $R_{34} = 0$, the condition $[D(e_2), e_4] + [e_2, D(e_4)] = 0$ is trivial. Again, the condition $[D(e_3), e_4] + [e_3, D(e_4)] = 0$ forces $R_{14} = R_{24} = 0$. We have therefore established the theorem's stated results.
\end{proof}
\subsection{The Lie algebra $\mathfrak{g}_{3.4}^0 \oplus \mathfrak{g}_1$}
The Lie algebra $\mathfrak{g}_{3.4}^0 \oplus \mathfrak{g}_1$ has a basis $\mathcal{B} = \left\lbrace e_1, e_2, e_3, e_4\right\rbrace$ such that the nonzero Lie brackets are $$[e_2, e_3] = e_1, \qquad [e_3, e_1] = -e_2.$$
\begin{theorem}
A pseudo-Riemannian inner product on $\mathfrak{g}_{3.4}^0 \oplus \mathfrak{g}_1$ is an algebraic Ricci soliton if and only if its Ricci operator satisfies the following identities
\begin{multicols}{2}
\begin{itemize}
\item $R_{i1} = R_{i2} = 0, \; i = 3, 4$
\item $R_{i4} = 0, \; i = 1, 2, 3$
\item $R_{21} = R_{12}$
\item $R_{22} = R_{11}$
\end{itemize}
\end{multicols}
In this situation, the Ricci operator can be written as $\operatorname{Ric} = \eta I_4 + D$ where $\eta = R_{33}$ and
$$D = \begin{bmatrix}
R_{11} - R_{33} & R_{12} & R_{13} & 0\\
\\
R_{12} & R_{11} - R_{33} & R_{23} & 0\\
\\
0 & 0 & 0 & 0\\
\\
0 & 0 & R_{43} & R_{44} - R_{33}
\end{bmatrix}$$
is a derivation of $\mathfrak{g}_{3.4}^0 \oplus \mathfrak{g}_1$ with respect to the basis $\mathcal{B}$.
\end{theorem}
\begin{proof}
From the fact that $D$ in (\ref{derivation}) is a derivation of $\mathfrak{g}_{3.4}^0 \oplus \mathfrak{g}_1$, it follows that the following condition holds
\begin{eqnarray*}
& [D(e_2), e_3] + [e_2, D(e_3)] = D[e_2, e_3] \hspace{7.1cm}\\ \Leftrightarrow & \left[ R_{12}e_1 + \left( R_{22} - \eta\right)e_2 + R_{32}e_3 + R_{42}e_4, e_3\right] + \hspace{5.4cm}\\ & \left[ e_2, R_{13}e_1 + R_{23}e_2 + \left(R_{33} - \eta\right)e_3 + R_{43}e_4\right] = D(e_1) \hspace{4.3cm}\\
\Leftrightarrow& R_{12}e_2 + \left( R_{22} - \eta\right)e_1 + \left( R_{33} - \eta\right)e_1 = \left(R_{11} - \eta\right)e_1 + R_{21}e_2 + R_{31}e_3 + R_{41}e_4
\end{eqnarray*}
By identifying the components of $e_i$, we obtain that
$$\left\lbrace\begin{array}{lll}
R_{22} - \eta + R_{33} - \eta = R_{11} - \eta \Longrightarrow \eta = R_{22} + R_{33} - R_{11} \\
\\
R_{12} = R_{21}, \qquad R_{31} = R_{41} = 0
\end{array}\right.$$
Similarly, $D$ adheres to the following condition
\begin{eqnarray*}
& [D(e_3), e_1] + [e_3, D(e_1)] = D[e_3, e_1] \hspace{8.06cm}\\ \Leftrightarrow & \left[ R_{13}e_1 + R_{23}e_2 + \left(R_{33} - \eta\right)e_3 + R_{43}e_4, e_1\right] + \left[ e_3, \left( R_{11} - \eta\right)e_1 + R_{21}e_2\right] = -D(e_2)\\
\Leftrightarrow& -\left( R_{33} - \eta\right)e_2 - \left( R_{11} - \eta\right)e_2 - R_{21}e_1 = - R_{12}e_1 - \left( R_{22} - \eta\right)e_2 - R_{32}e_3 - R_{42}e_4
\end{eqnarray*}
By examining the components of $e_i$, we find that
$$\left\lbrace\begin{array}{lll}
R_{33} - \eta + R_{11} - \eta = R_{22} - \eta \Longrightarrow \eta = R_{33} + R_{11} - R_{22} \\
\\
R_{21} = R_{12}, \qquad R_{32} = R_{42} = 0
\end{array}\right.$$
We have $\eta = R_{22} + R_{33} - R_{11} = R_{33} + R_{11} - R_{22}$, thus $2\eta = \eta + \eta = 2R_{33}$. Hence $\eta = R_{33}$ and then $R_{22} = R_{11}$. Thus, the derivation $D$ becomes
$$D = \begin{bmatrix}
R_{11} - R_{33} & R_{12} & R_{13} & R_{14}\\
\\
R_{12} & R_{11} - R_{33} & R_{23} & R_{24}\\
\\
0 & 0 & 0 & R_{34}\\
\\
0 & 0 & R_{43} & R_{44} - R_{33}
\end{bmatrix}.$$
Now, it is easy to see that the condition $[D(e_1), e_2] + [e_1, D(e_2)] = 0$ is trivial. But the condition $[D(e_1), e_4] + [e_1, D(e_4)] = 0$ gives $R_{34} = 0$. Since $R_{34} = 0$, the condition $[D(e_2), e_4] + [e_2, D(e_4)] = 0$ is trivial. Again, the condition $[D(e_3), e_4] + [e_3, D(e_4)] = 0$ forces $R_{14} = R_{24} = 0$. It follows that the consequences of the theorem have been obtained.
\end{proof}
\subsection{The Lie algebra $\mathfrak{g}_{3.4}^{\alpha} \oplus \mathfrak{g}_1$}
The Lie algebra $\mathfrak{g}_{3.4}^{\alpha} \oplus \mathfrak{g}_1$ ($\alpha > 0$ and $\alpha \neq 1$) has a basis $\mathcal{B} = \left\lbrace e_1, e_2, e_3, e_4\right\rbrace$ such that the nonzero Lie brackets are $$[e_2, e_3] = e_1 - \alpha e_2, \qquad [e_3, e_1] = \alpha e_1 - e_2.$$
\begin{theorem}
A pseudo-Riemannian inner product on $\mathfrak{g}_{3.4}^{\alpha} \oplus \mathfrak{g}_1$ is an algebraic Ricci soliton if and only if its Ricci operator fulfills the following relations
\begin{multicols}{2}
\begin{itemize}
\item $R_{i1} = R_{i2} = 0, \; i = 3, 4$
\item $R_{i4} = 0, \; i = 1, 2, 3$
\item $R_{21} = R_{12}$
\item $R_{22} = R_{11}$
\end{itemize}
\end{multicols}
In this situation, the Ricci operator can be written as $\operatorname{Ric} = \eta I_4 + D$ where $\eta = R_{33}$ and
$$D = \begin{bmatrix}
R_{11} - R_{33} & R_{12} & R_{13} & 0\\
\\
R_{12} & R_{11} - R_{33} & R_{23} & 0\\
\\
0 & 0 & 0 & 0\\
\\
0 & 0 & R_{43} & R_{44} - R_{33}
\end{bmatrix}$$
is a derivation of $\mathfrak{g}_{3.4}^{\alpha} \oplus \mathfrak{g}_1$ with respect to the basis $\mathcal{B}$.
\end{theorem}
\begin{proof}
Since $D$ in (\ref{derivation}) is a derivation of $\mathfrak{g}_{3.4}^{\alpha} \oplus \mathfrak{g}_1$, it satisfies the following condition
\begin{eqnarray*}
& [D(e_2), e_3] + [e_2, D(e_3)] = D[e_2, e_3] \hspace{11.15cm}\\ \Leftrightarrow & \left[ R_{12}e_1 + \left( R_{22} - \eta\right)e_2 + R_{32}e_3 + R_{42}e_4, e_3\right] + \hspace{9.4cm}\\  & \left[ e_2, R_{13}e_1 + R_{23}e_2 + \left(R_{33} - \eta\right)e_3 + R_{43}e_4\right] = D(e_1) - \alpha D(e_2) \hspace{6.5cm}\\
\Leftrightarrow& R_{12}(e_2 - \alpha e_1) + \left( R_{22} - \eta\right)(e_1 - \alpha e_2) + \left( R_{33} - \eta\right)(e_1 - \alpha e_2) \hspace{6.7cm}\\ & = \left(R_{11} - \eta - \alpha R_{12}\right)e_1 + \left(R_{21} - \alpha\left( R_{22} - \eta\right)\right)e_2 + \left(R_{31} - \alpha R_{32}\right)e_3 + \left(R_{41} - \alpha R_{42}\right)e_4\hspace{2.4cm}
\end{eqnarray*}
Identifying the components of $e_i$ yields
$$\left\lbrace\begin{array}{lll}
-\alpha R_{12} + R_{22} - \eta + R_{33} - \eta = R_{11} - \eta - \alpha R_{12} \Longrightarrow \eta = R_{22} + R_{33} - R_{11} \\
\\
R_{12} - \alpha \left(R_{22} - \eta\right) - \alpha\left(R_{33} - \eta\right) = R_{21} - \alpha\left(R_{22} - \eta\right) \Longrightarrow \alpha\left(R_{33} - \eta\right) = R_{12} - R_{21} \\
\\
R_{31} = \alpha R_{32}, \quad R_{41} = \alpha R_{42}
\end{array}\right.$$
Similarly, $D$ satisfies the following condition
\begin{eqnarray*}
& [D(e_3), e_1] + [e_3, D(e_1)] = D[e_3, e_1] \hspace{14cm}\\ \Leftrightarrow & \left[ R_{13}e_1 + R_{23}e_2 + \left(R_{33} - \eta\right)e_3 + R_{43}e_4, e_1\right] + \hspace{12.3cm}\\  & \left[e_3, \left(R_{11} - \eta\right)e_1 + R_{21}e_2 + R_{31}e_3 + R_{41}e_4\right] = \alpha D(e_1) - D(e_2) \hspace{9.4cm}\\
\Leftrightarrow& \left( R_{33} - \eta\right)(\alpha e_1 - e_2) + \left( R_{11} - \eta\right)(\alpha e_1 - e_2) + R_{21}(\alpha e_2 - e_1) = \hspace{9cm}\\ &  \left(\alpha\left(R_{11} - \eta\right) - R_{12}\right)e_1 + \left(\alpha R_{21} - \left( R_{22} - \eta\right)\right)e_2 + \left(\alpha R_{31} - R_{32}\right)e_3 + \left(\alpha R_{41} - R_{42}\right)e_4\hspace{5.3cm}
\end{eqnarray*}
From the components of $e_i$, it follows that
$$\left\lbrace\begin{array}{lll}
\alpha\left(R_{33} - \eta\right) + \alpha\left(R_{11} - \eta\right) - R_{21} = \alpha\left(R_{11} - \eta\right) - R_{12} \Longrightarrow \alpha\left(R_{33} - \eta\right) = R_{21} - R_{12}  \\
\\
-\left(R_{33} - \eta\right) - \left(R_{11} - \eta\right) + \alpha R_{21} = \alpha R_{21} - \left(R_{22} - \eta\right) \Longrightarrow \eta = R_{33} + R_{11} - R_{22}  \\
\\
\alpha R_{31} = R_{32}, \quad \alpha R_{41} = R_{42}
\end{array}\right.$$
Hence $\alpha R_{31} = \alpha^2 R_{32} = R_{32}$. Since $\alpha > 0$ and $\alpha \neq 1$, then $R_{32} = 0$ which implies that $R_{31} = 0$. Similarly, we see that $\alpha R_{41} = \alpha^2 R_{42} = R_{42}$. Then $R_{42} = R_{41} = 0$. Once again, we see that: $\alpha\left(R_{33} - \eta\right) = R_{12} - R_{21} = -\left(R_{12} - R_{21}\right)$. Thus, $\eta = R_{33}$ and $R_{21} = R_{12}$. Since $\eta =  R_{33} = R_{22} + R_{33} - R_{11} = R_{33} + R_{11} - R_{22}$, then $R_{22} = R_{11}$. Therefore, the derivation $D$ becomes
$$D = \begin{bmatrix}
R_{11} - R_{33} & R_{12} & R_{13} & R_{14}\\
\\
R_{12} & R_{11} - R_{33} & R_{23} & R_{24}\\
\\
0 & 0 & 0 & R_{34}\\
\\
0 & 0 & R_{43} & R_{44} - R_{33}
\end{bmatrix}.$$
Now, it is easy to see that the condition $[D(e_1), e_2] + [e_1, D(e_2)] = 0$ is trivial. But the condition $[D(e_1), e_4] + [e_1, D(e_4)] = 0$ gives $R_{34} = 0$. Since $R_{34} = 0$, the condition $[D(e_2), e_4] + [e_2, D(e_4)] = 0$ is trivial. Next, we can see that
\begin{eqnarray*}
\left[D(e_3), e_4\right] + \left[e_3, D(e_4)\right] = 0 &\Leftrightarrow& \left[e_3, R_{14}e_1 + R_{24}e_2 + \left(R_{44} - R_{33}\right)e_4\right] = 0\\
&\Leftrightarrow& \alpha R_{14}e_1 - R_{14}e_2 + \alpha R_{24}e_2 - R_{24}e_1 = 0\\
&\Leftrightarrow& \alpha R_{14} = R_{24}, \qquad \alpha R_{24} = R_{14}
\end{eqnarray*}
Hence $\alpha R_{14} = \alpha^2 R_{24} = R_{24}$. Since $\alpha > 0$ and $\alpha \neq 1$, then $R_{24} = R_{14} = 0$. Thus, the theorem's conclusions follow.
\end{proof}
\subsection{The Lie algebra $\mathfrak{g}_{3.5}^0 \oplus \mathfrak{g}_1$}
The Lie algebra $\mathfrak{g}_{3.5}^0 \oplus \mathfrak{g}_1$ has a basis $\mathcal{B} = \left\lbrace e_1, e_2, e_3, e_4\right\rbrace$ such that the nonzero Lie brackets are $$[e_2, e_3] = e_1, \qquad [e_3, e_1] = e_2.$$
\begin{theorem}
A pseudo-Riemannian inner product on $\mathfrak{g}_{3.5}^0 \oplus \mathfrak{g}_1$ is an algebraic Ricci soliton if and only if its Ricci operator satisfies the conditions below
\begin{multicols}{2}
\begin{itemize}
\item $R_{i1} = R_{i2} = 0, \; i = 3, 4$
\item $R_{i4} = 0, \; i = 1, 2, 3$
\item $R_{21} = -R_{12}$
\item $R_{22} = R_{11}$
\end{itemize}
\end{multicols}
With these conditions in place, the Ricci operator may be written as $\operatorname{Ric} = \eta I_4 + D$ where $\eta = R_{33}$ and
$$D = \begin{bmatrix}
R_{11} - R_{33} & R_{12} & R_{13} & 0\\
\\
-R_{12} & R_{11} - R_{33} & R_{23} & 0\\
\\
0 & 0 & 0 & 0\\
\\
0 & 0 & R_{43} & R_{44} - R_{33}
\end{bmatrix}$$
is a derivation of $\mathfrak{g}_{3.5}^0 \oplus \mathfrak{g}_1$ with respect to the basis $\mathcal{B}$.
\end{theorem}
\begin{proof}
Given that $D$ in (\ref{derivation}) is a derivation of $\mathfrak{g}_{3.5}^0 \oplus \mathfrak{g}_1$, the following condition holds
\begin{eqnarray*}
& [D(e_2), e_3] + [e_2, D(e_3)] = D[e_2, e_3] \hspace{7.6cm}\\ \Leftrightarrow & \left[ R_{12}e_1 + \left( R_{22} - \eta\right)e_2 + R_{32}e_3 + R_{42}e_4, e_3\right] + \hspace{5.8cm}\\ & \left[ e_2, R_{13}e_1 + R_{23}e_2 + \left(R_{33} - \eta\right)e_3 + R_{43}e_4\right] = D(e_1) \hspace{4.7cm}\\
\Leftrightarrow& -R_{12}e_2 + \left( R_{22} - \eta\right)e_1 + \left( R_{33} - \eta\right)e_1 = \left(R_{11} - \eta\right)e_1 + R_{21}e_2 + R_{31}e_3 + R_{41}e_4
\end{eqnarray*}
An analysis of the components of $e_i$ shows that
$$\left\lbrace\begin{array}{lll}
R_{22} - \eta + R_{33} - \eta = R_{11} - \eta \Rightarrow \eta = R_{22} + R_{33} - R_{11} \\
\\
R_{21} = -R_{12}, \qquad R_{31} = R_{41} = 0
\end{array}\right.$$
Then the derivation $D$ becomes
$$D = \begin{bmatrix}
R_{11} - \eta & R_{12} & R_{13} & R_{14}\\
\\
-R_{12} & R_{22} - \eta & R_{23} & R_{24}\\
\\
0 & R_{32} & R_{33} - \eta & R_{34}\\
\\
0 & R_{42} & R_{43} & R_{44} - \eta
\end{bmatrix}.$$
Once again, the following condition is satisfied by $D$
\begin{eqnarray*}
& [D(e_3), e_1] + [e_3, D(e_1)] = D[e_3, e_1] \hspace{7.7cm}\\ \Leftrightarrow & \left[ R_{13}e_1 + R_{23}e_2 + \left(R_{33} - \eta\right)e_3 + R_{43}e_4, e_1\right] + \left[ e_3, \left( R_{11} - \eta\right)e_1 - R_{12}e_2\right] = D(e_2) \\
\Leftrightarrow& \left( R_{33} - \eta\right)e_2 + \left( R_{11} - \eta\right)e_2 + R_{12}e_1 = R_{12}e_1 + \left( R_{22} - \eta\right)e_2 + R_{32}e_3 + R_{42}e_4\hspace{0.5cm}
\end{eqnarray*}
Upon identifying the components of $e_i$, we conclude that
$$\left\lbrace\begin{array}{lll}
R_{33} - \eta + R_{11} - \eta = R_{22} - \eta \Rightarrow \eta = R_{33} + R_{11} - R_{22} \\
\\
R_{32} = R_{42} = 0
\end{array}\right.$$
We have $\eta = R_{22} + R_{33} - R_{11} = R_{33} + R_{11} - R_{22}$, thus $2\eta = \eta + \eta = 2R_{33}$. Hence $\eta = R_{33}$ and then $R_{22} = R_{11}$. Thus, the derivation $D$ becomes
$$D = \begin{bmatrix}
R_{11} - R_{33} & R_{12} & R_{13} & R_{14}\\
\\
-R_{12} & R_{11} - R_{33} & R_{23} & R_{24}\\
\\
0 & 0 & 0 & R_{34}\\
\\
0 & 0 & R_{43} & R_{44} - R_{33}
\end{bmatrix}.$$
Now, it is easy to see that the condition $[D(e_1), e_2] + [e_1, D(e_2)] = 0$ is trivial. But the condition $[D(e_1), e_4] + [e_1, D(e_4)] = 0$ gives $R_{34} = 0$. Since $R_{34} = 0$, the rule $[D(e_2), e_4] + [e_2, D(e_4)] = 0$ is trivial. Again, the rule $[D(e_3), e_4] + [e_3, D(e_4)] = 0$ forces $R_{14} = R_{24} = 0$. Hence, the theorem's results are obtained.
\end{proof}
\subsection{The Lie algebra $\mathfrak{g}_{3.5}^{\alpha} \oplus \mathfrak{g}_1$}
The Lie algebra $\mathfrak{g}_{3.5}^{\alpha} \oplus \mathfrak{g}_1$ ($\alpha > 0$) has a basis $\mathcal{B} = \left\lbrace e_1, e_2, e_3, e_4\right\rbrace$ such that the nonzero Lie brackets are $$[e_2, e_3] = e_1 - \alpha e_2, \qquad [e_3, e_1] = \alpha e_1 + e_2.$$
\begin{theorem}
A pseudo-Riemannian inner product on $\mathfrak{g}_{3.5}^{\alpha} \oplus \mathfrak{g}_1$ is an algebraic Ricci soliton if and only if its associated Ricci operator satisfies the conditions
\begin{multicols}{2}
\begin{itemize}
\item $R_{31} = R_{32} = R_{i4} = 0, \; i = 1, 2, 3$
\item $R_{21} = -R_{12}$
\item $R_{42} = \alpha R_{41}$
\item $R_{22} = R_{11}$
\end{itemize}
\end{multicols}
Under the stated conditions, the Ricci operator can be expressed as $\operatorname{Ric} = \eta I_4 + D$ where $\eta = R_{33}$ and
$$D = \begin{bmatrix}
R_{11} - R_{33} & R_{12} & R_{13} & 0\\
\\
-R_{12} & R_{11} - R_{33} & R_{23} & 0\\
\\
0 & 0 & 0 & 0\\
\\
R_{41} & \alpha R_{41} & R_{43} & R_{44} - R_{33}
\end{bmatrix}$$
is a derivation of $\mathfrak{g}_{3.5}^{\alpha} \oplus \mathfrak{g}_1$ with respect to the basis $\mathcal{B}$. Note that if $R_{41}  = 0$, there is no additional condition over $\alpha$. But if $R_{41} \neq 0$, then $\alpha = 1$.
\end{theorem}
\begin{proof}
Being a derivation of $\mathfrak{g}_{3.5}^{\alpha} \oplus \mathfrak{g}_1$, $D$ in (\ref{derivation}) satisfies the following condition
\begin{eqnarray*}
& [D(e_2), e_3] + [e_2, D(e_3)] = D[e_2, e_3] \hspace{9.25cm}\\ \Leftrightarrow & \left[ R_{12}e_1 + \left( R_{22} - \eta\right)e_2 + R_{32}e_3 + R_{42}e_4, e_3\right] + \hspace{7.5cm}\\  & \left[ e_2, R_{13}e_1 + R_{23}e_2 + \left(R_{33} - \eta\right)e_3 + R_{43}e_4\right] = D(e_1) - \alpha D(e_2) \hspace{4.6cm}\\
\Leftrightarrow&  R_{12}(-\alpha e_1 - e_2) + \left( R_{22} - \eta\right)(e_1 - \alpha e_2) + \left( R_{33} - \eta\right)(e_1 - \alpha e_2) = \hspace{4cm}\\ & \left(R_{11} - \eta - \alpha R_{12}\right)e_1 + \left(R_{21} - \alpha\left( R_{22} - \eta\right)\right)e_2 + \left(R_{31} - \alpha R_{32}\right)e_3 + \left(R_{41} - \alpha R_{42}\right)e_4\hspace{1cm}
\end{eqnarray*}
From the components of $e_i$, we obtain
$$\left\lbrace\begin{array}{lll}
-\alpha R_{12} + R_{22} - \eta + R_{33} - \eta = R_{11} - \eta - \alpha R_{12} \Longrightarrow \eta = R_{22} + R_{33} - R_{11} \\
\\
-R_{12} - \alpha \left(R_{22} - \eta\right) - \alpha\left(R_{33} - \eta\right) = R_{21} - \alpha\left(R_{22} - \eta\right) \Longrightarrow  \alpha\left(R_{33} - \eta\right) = -(R_{12} + R_{21})\\
\\
R_{31} = \alpha R_{32}, \qquad R_{41} = \alpha R_{42}
\end{array}\right.$$
Similarly, $D$ satisfies the following condition
\begin{eqnarray*}
& [D(e_3), e_1] + [e_3, D(e_1)] = D[e_3, e_1] \hspace{9.6cm}\\ \Leftrightarrow & \left[ R_{13}e_1 + R_{23}e_2 + \left(R_{33} - \eta\right)e_3 + R_{43}e_4, e_1\right] + \hspace{7.8cm}\\  & \left[e_3, \left(R_{11} - \eta\right)e_1 + R_{21}e_2 + R_{31}e_3 + R_{41}e_4\right] = \alpha D(e_1) + D(e_2) \hspace{4.9cm}\\
\Leftrightarrow& \left( R_{33} - \eta\right)(\alpha e_1 + e_2) + \left( R_{11} - \eta\right)(\alpha e_1 + e_2) + R_{21}(\alpha e_2 - e_1) = \hspace{4.5cm}\\ &  \left(\alpha\left(R_{11} - \eta\right) + R_{12}\right)e_1 + \left(\alpha R_{21} + \left( R_{22} - \eta\right)\right)e_2 + \left(\alpha R_{31} + R_{32}\right)e_3 + \left(\alpha R_{41} + R_{42}\right)e_4\hspace{0.7cm}
\end{eqnarray*}
By identifying each component of $e_i$, we deduce that
$$\left\lbrace\begin{array}{lll}
\alpha\left(R_{33} - \eta\right) + \alpha\left(R_{11} - \eta\right) - R_{21} = \alpha\left(R_{11} - \eta\right) + R_{12} \Longrightarrow \alpha\left(R_{33} - \eta\right) = R_{12} + R_{21}  \\
\\
R_{33} - \eta + R_{11} - \eta + \alpha R_{21} = \alpha R_{21} + R_{22} - \eta \Longrightarrow \eta = R_{33} + R_{11} - R_{22}  \\
\\
\alpha R_{31} = R_{32}, \qquad \alpha R_{41} = R_{42}
\end{array}\right.$$
From the two bracket conditions as above, we see that
$$\alpha\left(R_{33} - \eta\right) = R_{12} + R_{21} = -(R_{12} + R_{21}) \Longrightarrow R_{21} = -R_{12} \quad\text{and}\quad \eta = R_{33}.$$
Since $\eta = R_{33} = R_{33} + R_{11} - R_{22} = R_{22} + R_{33} - R_{11}$, then $R_{22} = R_{11}$.
Then the derivation $D$ becomes
$$D = \begin{bmatrix}
R_{11} - R_{33} & R_{12} & R_{13} & R_{14}\\
\\
-R_{12} & R_{11} - R_{33} & R_{23} & R_{24}\\
\\
R_{31} & R_{32} & 0 & R_{34}\\
\\
R_{41} & R_{42} & R_{43} & R_{44} - R_{33}
\end{bmatrix}.$$
Once again, the following condition holds
\begin{eqnarray*}
& [D(e_1), e_2] + [e_1, D(e_2)] = D[e_1, e_2] \hspace{2.5cm}\\ \Leftrightarrow & \left[ \left(R_{11} - R_{33}\right)e_1 - R_{12}e_2 + R_{31}e_3 + R_{41}e_4, e_2\right] + \hspace{0.3cm}\\  & \left[e_1, R_{12}e_1 + \left(R_{11} - R_{33}\right)e_2 + R_{32}e_3 + R_{42}e_4\right] = 0 \\
\Leftrightarrow& R_{31}(\alpha e_2 - e_1) + R_{32}(-\alpha e_1 - e_2) = 0\hspace{2cm}\\
\Leftrightarrow& \left( -R_{31} - \alpha R_{32}\right)e_1 + \left( \alpha R_{31} - R_{32}\right)e_2 = 0\hspace{1.2cm}
\end{eqnarray*}
This shows that $R_{31} = -\alpha R_{32}$ and $\alpha R_{31} = R_{32}$. Hence $\alpha R_{31} = -\alpha^2 R_{32} = R_{32}$, thus $R_{32} = R_{31} = 0$. Now, it is easy to see that the property $[D(e_1), e_4] + [e_1, D(e_4)] = 0$ gives $R_{34} = 0$. Since $R_{34} = 0$, the property $[D(e_2), e_4] + [e_2, D(e_4)] = 0$ is trivial. Next, $D$ satisfies the following rule
\begin{eqnarray*}
[D(e_3), e_4] + [e_3, D(e_4)] = D[e_3, e_4] &\Leftrightarrow& \left[e_3, R_{14}e_1 + R_{24}e_2 + \left(R_{44} - R_{33}\right)e_4\right] = 0 \\
&\Leftrightarrow& R_{14}(\alpha e_1 + e_2) + R_{24}(\alpha e_2 - e_1) = 0\\
&\Leftrightarrow& \left( \alpha R_{14} - R_{24}\right)e_1 + \left( R_{14} + \alpha R_{24}\right)e_2 = 0
\end{eqnarray*}
This shows that $\alpha R_{14} = R_{24}$ and $R_{14} = -\alpha R_{24}$. Hence $\alpha R_{14} = -\alpha^2 R_{24} = R_{24}$, thus $R_{24} = R_{14} = 0$. Note that from the two first bracket conditions, we get that
$$\alpha R_{42} = \alpha^2R_{41} = R_{41}.$$
If $R_{41} = 0$, there is no additional condition over $\alpha$. But if $R_{41} \neq 0$, then $\alpha = 1$. Therefore, we have reached the results guaranteed by the theorem.
\end{proof}
\subsection{The Lie algebra $\mathfrak{g}_{3.6} \oplus \mathfrak{g}_1$}
The Lie algebra $\mathfrak{g}_{3.6} \oplus \mathfrak{g}_1$ has a basis $\mathcal{B} = \left\lbrace e_1, e_2, e_3, e_4\right\rbrace$ such that the nonzero Lie brackets are $$[e_2, e_3] = e_1, \qquad [e_3, e_1] = e_2, \qquad [e_1, e_2] = -e_3.$$
\begin{theorem}
A pseudo-Riemannian inner product on $\mathfrak{g}_{3.6} \oplus \mathfrak{g}_1$ is an algebraic Ricci soliton if and only if its Ricci operator satisfies the following identities
\begin{multicols}{2}
\begin{itemize}
\item $R_{i4} = R_{4i} = 0, \; i = 1, 2, 3$
\item $R_{21} = -R_{12}$
\item $R_{31} = R_{13}$
\item $R_{32} = R_{23}$
\item $R_{33} = R_{22} = R_{11}$
\end{itemize}
\end{multicols}
Provided these identities are satisfied, the Ricci operator can take the form\\ $\operatorname{Ric} = \eta I_4 + D$ where $\eta = R_{33} = R_{22} = R_{11}$ and
$$D = \begin{bmatrix}
0 & R_{12} & R_{13} & 0\\
\\
-R_{12} & 0 & R_{23} & 0\\
\\
R_{13} & R_{23} & 0 & 0\\
\\
0 & 0 & 0 & R_{44} - R_{11}
\end{bmatrix}$$
is a derivation of $\mathfrak{g}_{3.6} \oplus \mathfrak{g}_1$ with respect to the basis $\mathcal{B}$.
\end{theorem}
\begin{proof}
As $D$ in (\ref{derivation}) serves as a derivation of $\mathfrak{g}_{3.6} \oplus \mathfrak{g}_1$, the following property holds
\begin{eqnarray*}
& [D(e_2), e_3] + [e_2, D(e_3)] = D[e_2, e_3] \hspace{2.8cm}\\ \Leftrightarrow & \left[ R_{12}e_1 + \left( R_{22} - \eta\right)e_2 + R_{32}e_3 + R_{42}e_4, e_3\right] + \hspace{1.1cm}\\ & \left[ e_2, R_{13}e_1 + R_{23}e_2 + \left(R_{33} - \eta\right)e_3 + R_{43}e_4\right] = D(e_1) \\
\Leftrightarrow&  -R_{12}e_2 + \left(R_{22} - \eta\right)e_1 + R_{13}e_3 + \left( R_{33} - \eta\right)e_1 \hspace{0.9cm}\\ & = \left(R_{11} - \eta\right)e_1 + R_{21}e_2 + R_{31}e_3 + R_{41}e_4\hspace{1.9cm}
\end{eqnarray*}
By identifying the components of $e_i$, we obtain that
$$\left\lbrace\begin{array}{lll}
R_{22} - \eta + R_{33} - \eta = R_{11} - \eta \Longrightarrow \eta = R_{22} + R_{33} - R_{11} \\
\\
-R_{12} = R_{21},\qquad R_{13} = R_{31}, \qquad R_{41} = 0
\end{array}\right.$$
As before, $D$ meets the following property
\begin{eqnarray*}
& [D(e_3), e_1] + [e_3, D(e_1)] = D[e_3, e_1] \hspace{2.8cm}\\ \Leftrightarrow & \left[ R_{13}e_1 + R_{23}e_2 + \left( R_{33} - \eta\right)e_3 + R_{43}e_4, e_1\right] + \hspace{1.1cm}\\ & \left[ e_3, \left(R_{11} - \eta\right)e_1 + R_{21}e_2 + R_{31}e_3 + R_{41}e_4\right] = D(e_2) \\
\Leftrightarrow&  R_{23}e_3 + \left(R_{33} - \eta\right)e_2 + \left( R_{11} - \eta\right)e_2 - R_{21}e_1 \hspace{1.2cm}\\ & = R_{12}e_1 +  \left(R_{22} - \eta\right)e_2 + R_{32}e_3 + R_{42}e_4\hspace{1.9cm}
\end{eqnarray*}
By examining the components of $e_i$, we find that
$$\left\lbrace\begin{array}{lll}
R_{33} - \eta + R_{11} - \eta = R_{22} - \eta \Longrightarrow \eta = R_{33} + R_{11} - R_{22} \\
\\
-R_{21} = R_{12},\qquad R_{23} = R_{32}, \qquad R_{42} = 0
\end{array}\right.$$
The following rule is once again satisfied by $D$
\begin{eqnarray*}
& [D(e_1), e_2] + [e_1, D(e_2)] = D[e_1, e_2] \hspace{3.15cm}\\ \Leftrightarrow & \left[ \left( R_{11} - \eta\right)e_1 + R_{21}e_2 + R_{31}e_3 + R_{41}e_4, e_2\right] + \hspace{1.43cm}\\ & \left[ e_1, R_{12}e_1 + \left( R_{22} - \eta\right)e_2 + R_{32}e_3 + R_{42}e_4\right] = -D(e_3) \\
\Leftrightarrow&  -\left(R_{11} - \eta\right)e_3 - R_{31}e_1 - \left( R_{22} - \eta\right)e_3 - R_{32}e_2 \hspace{1.2cm}\\ & = -R_{13}e_1 - R_{23}e_2 -  \left(R_{33} - \eta\right)e_3 - R_{43}e_4\hspace{1.9cm}
\end{eqnarray*}
Identifying the components of $e_i$ yields
$$\left\lbrace\begin{array}{lll}
R_{11} - \eta + R_{22} - \eta = R_{33} - \eta \Rightarrow \eta = R_{11} + R_{22} - R_{33} \\
\\
R_{31} = R_{13},\qquad R_{32} = R_{23}, \qquad R_{43} = 0
\end{array}\right.$$
 The three conditions over the constant $\eta$ are
 \begin{align}
 \eta &= R_{22} + R_{33} - R_{11} \label{eq1} \\
 \eta &= R_{33} + R_{11} - R_{22} \label{eq2} \\
 \eta &= R_{11} + R_{22} - R_{33} \label{eq3} 
 \end{align}
 We see that $(\ref{eq1}) + (\ref{eq2}) \Rightarrow \eta = R_{33}$, then $R_{22} = R_{11}$. Similarly $(\ref{eq1}) + (\ref{eq3}) \Rightarrow \eta = R_{22}$, then $R_{33} = R_{11}$. Hence $\eta = R_{33} = R_{22} = R_{11}$. Next, the following condition holds
 \begin{eqnarray*}
 \left[ D(e_1), e_4\right] + \left[ e_1, D(e_4)\right] = 0 &\Leftrightarrow& \left[ e_1, R_{14}e_1 + R_{24}e_2 + R_{34}e_3 + (R_{44} - \eta)e_4\right] = 0\\
 &\Leftrightarrow& -R_{24}e_3 - R_{34}e_2 = 0 \Leftrightarrow R_{24} = R_{34} = 0
 \end{eqnarray*}
 Similarly, the condition $ \left[ D(e_2), e_4\right] + \left[ e_2, D(e_4)\right] = 0$ gives $R_{14} = 0$. Since $R_{14} = R_{24} = R_{34} = 0$, then the condition $ \left[ D(e_3), e_4\right] + \left[ e_3, D(e_4)\right] = 0$ is trivial. As a consequence, the outcomes described in the theorem are now established.
\end{proof}
\subsection{The Lie algebra $\mathfrak{g}_{3.7} \oplus \mathfrak{g}_1$}
The Lie algebra $\mathfrak{g}_{3.7} \oplus \mathfrak{g}_1$ has a basis $\mathcal{B} = \left\lbrace e_1, e_2, e_3, e_4\right\rbrace$ such that the nonzero Lie brackets are $$[e_2, e_3] = e_1, \qquad [e_3, e_1] = e_2, \qquad [e_1, e_2] = e_3.$$
\begin{theorem}
A pseudo-Riemannian inner product on $\mathfrak{g}_{3.7} \oplus \mathfrak{g}_1$ is an algebraic Ricci soliton if and only if its Ricci operator satisfies
\begin{multicols}{2}
\begin{itemize}
\item $R_{i4} = R_{4i} = 0, \; i = 1, 2, 3$
\item $R_{21} = -R_{12}$
\item $R_{31} = -R_{13}$
\item $R_{32} = -R_{23}$
\item $R_{33} = R_{22} = R_{11}$
\end{itemize}
\end{multicols}
Under the stated conditions, the Ricci operator can be expressed as $\operatorname{Ric} = \eta I_4 + D$ where $\eta = R_{33} = R_{22} = R_{11}$ and
$$D = \begin{bmatrix}
0 & R_{12} & R_{13} & 0\\
\\
-R_{12} & 0 & R_{23} & 0\\
\\
-R_{13} & -R_{23} & 0 & 0\\
\\
0 & 0 & 0 & R_{44} - R_{11}
\end{bmatrix}$$
is a derivation of $\mathfrak{g}_{3.7} \oplus \mathfrak{g}_1$ with respect to the basis $\mathcal{B}$.
\end{theorem}
\begin{proof}
As $D$ in (\ref{derivation}) serves as a derivation of $\mathfrak{g}_{3.7} \oplus \mathfrak{g}_1$, the following rule holds
\begin{eqnarray*}
& [D(e_2), e_3] + [e_2, D(e_3)] = D[e_2, e_3] \hspace{2.8cm}\\ \Leftrightarrow & \left[ R_{12}e_1 + \left( R_{22} - \eta\right)e_2 + R_{32}e_3 + R_{42}e_4, e_3\right] + \hspace{1.1cm}\\ & \left[ e_2, R_{13}e_1 + R_{23}e_2 + \left(R_{33} - \eta\right)e_3 + R_{43}e_4\right] = D(e_1) \\
\Leftrightarrow&  -R_{12}e_2 + \left(R_{22} - \eta\right)e_1 - R_{13}e_3 + \left( R_{33} - \eta\right)e_1 \hspace{0.9cm}\\ & = \left(R_{11} - \eta\right)e_1 + R_{21}e_2 + R_{31}e_3 + R_{41}e_4\hspace{1.9cm}
\end{eqnarray*}
From the components of $e_i$, it follows that
$$\left\lbrace\begin{array}{lll}
R_{22} - \eta + R_{33} - \eta = R_{11} - \eta \Longrightarrow \eta = R_{22} + R_{33} - R_{11} \\
\\
-R_{12} = R_{21},\qquad -R_{13} = R_{31}, \qquad R_{41} = 0
\end{array}\right.$$
Similarly, the following rule holds
\begin{eqnarray*}
& [D(e_3), e_1] + [e_3, D(e_1)] = D[e_3, e_1] \hspace{2.8cm}\\ \Leftrightarrow & \left[ R_{13}e_1 + R_{23}e_2 + \left( R_{33} - \eta\right)e_3 + R_{43}e_4, e_1\right] + \hspace{1.1cm}\\ & \left[ e_3, \left(R_{11} - \eta\right)e_1 + R_{21}e_2 + R_{31}e_3 + R_{41}e_4\right] = D(e_2) \\
\Leftrightarrow&  -R_{23}e_3 + \left(R_{33} - \eta\right)e_2 + \left( R_{11} - \eta\right)e_2 - R_{21}e_1 \hspace{0.9cm}\\ & = R_{12}e_1 +  \left(R_{22} - \eta\right)e_2 + R_{32}e_3 + R_{42}e_4\hspace{1.9cm}
\end{eqnarray*}
An analysis of the components of $e_i$ shows that
$$\left\lbrace\begin{array}{lll}
R_{33} - \eta + R_{11} - \eta = R_{22} - \eta \Longrightarrow \eta = R_{33} + R_{11} - R_{22} \\
\\
-R_{21} = R_{12},\qquad -R_{23} = R_{32}, \qquad R_{42} = 0
\end{array}\right.$$
The derivation $D$ satisfies again the following rule
\begin{eqnarray*}
& [D(e_1), e_2] + [e_1, D(e_2)] = D[e_1, e_2] \hspace{2.8cm}\\ \Leftrightarrow & \left[ \left( R_{11} - \eta\right)e_1 + R_{21}e_2 + R_{31}e_3 + R_{41}e_4, e_2\right] + \hspace{1.1cm}\\ & \left[ e_1, R_{12}e_1 + \left( R_{22} - \eta\right)e_2 + R_{32}e_3 + R_{42}e_4\right] = D(e_3) \\
\Leftrightarrow&  \left(R_{11} - \eta\right)e_3 - R_{31}e_1 + \left( R_{22} - \eta\right)e_3 - R_{32}e_2 \hspace{1.2cm}\\ & = R_{13}e_1 + R_{23}e_2 + \left(R_{33} - \eta\right)e_3 + R_{43}e_4\hspace{1.9cm}
\end{eqnarray*}
Upon identifying the components of $e_i$, we conclude that
$$\left\lbrace\begin{array}{lll}
R_{11} - \eta + R_{22} - \eta = R_{33} - \eta \Rightarrow \eta = R_{11} + R_{22} - R_{33} \\
\\
-R_{31} = R_{13},\qquad -R_{32} = R_{23}, \qquad R_{43} = 0
\end{array}\right.$$
The three conditions over the constant $\eta$ are
\begin{align}
\eta &= R_{22} + R_{33} - R_{11} \label{eqq1} \\
\eta &= R_{33} + R_{11} - R_{22} \label{eqq2} \\
\eta &= R_{11} + R_{22} - R_{33} \label{eqq3} 
\end{align}
From $(\ref{eqq1}) + (\ref{eqq2})$ we obtain $\eta = R_{33}$, and consequently $R_{22} = R_{11}$. Analogously, $(\ref{eqq1}) + (\ref{eqq3})$ gives $\eta = R_{22}$, leading to $R_{33} = R_{11}$. Hence $\eta = R_{33} = R_{22} = R_{11}$. Next, the following condition holds
\begin{eqnarray*}
\left[ D(e_1), e_4\right] + \left[ e_1, D(e_4)\right] = 0 &\Leftrightarrow& \left[ e_1, R_{14}e_1 + R_{24}e_2 + R_{34}e_3 + (R_{44} - \eta)e_4\right] = 0\\
&\Leftrightarrow& R_{24}e_3 - R_{34}e_2 = 0 \Leftrightarrow R_{24} = R_{34} = 0
\end{eqnarray*}
In a similar manner, the property $ \left[ D(e_2), e_4\right] + \left[ e_2, D(e_4)\right] = 0$ gives $R_{14} = 0$. Since $R_{14} = R_{24} = R_{34} = 0$, then the property $ \left[ D(e_3), e_4\right] + \left[ e_3, D(e_4)\right] = 0$ is trivial. Thus, we have obtained the conclusions of the theorem.
\end{proof}
\section{Pseudo-Riemannian algebraic Ricci solitons on indecomposable four-dimensional Lie algebras}
\subsection{The Lie algebra $\mathfrak{g}_{4.1}$}
The Lie algebra $\mathfrak{g}_{4.1}$ has a basis $\mathcal{B} = \left\lbrace e_1, e_2, e_3, e_4\right\rbrace$ such that the nonzero Lie brackets are 
$$[e_2, e_4] = e_1, \qquad [e_3, e_4] = e_2.$$
\begin{theorem}
A pseudo-Riemannian inner product on $\mathfrak{g}_{4.1}$ is an algebraic Ricci soliton if and only if its Ricci operator fulfills the following relations
\begin{multicols}{2}
\begin{itemize}
\item $R_{i1} = 0, \; i = 2, 3, 4$
\item $R_{32} = R_{42} = R_{43} = 0$
\item $R_{23} = R_{12}$
\item $R_{22} + R_{44} - R_{11} = R_{33} + R_{44} - R_{22}$
\end{itemize}
\end{multicols}
Given these relations, the Ricci operator can be described as $\operatorname{Ric} = \eta I_4 + D$ where $\eta = R_{22} + R_{44} - R_{11} = R_{33} + R_{44} - R_{22}$ and
$$D = \begin{bmatrix}
R_{11} - \eta & R_{12} & R_{13} & R_{14}\\
\\
0 & R_{22} - \eta & R_{12} & R_{24}\\
\\
0 & 0 & R_{33} - \eta & R_{34}\\
\\
0 & 0 & 0 & R_{44} - \eta
\end{bmatrix}.$$
is a derivation of $\mathfrak{g}_{4.1}$ with respect to the basis $\mathcal{B}$.
\end{theorem}
\begin{proof}
Since $D$ in (\ref{derivation}) is a derivation of $\mathfrak{g}_{4.1}$, the condition below holds
\begin{eqnarray*}
& [D(e_2), e_4] + [e_2, D(e_4)] = D[e_2, e_4] \hspace{7.2cm}\\ \Leftrightarrow & \left[ R_{12}e_1 + \left( R_{22} - \eta\right)e_2 + R_{32}e_3 + R_{42}e_4, e_4\right] + \hspace{5.4cm}\\ & \left[ e_2, R_{14}e_1 + R_{24}e_2 + R_{34}e_3 + \left(R_{44} - \eta\right)e_4\right] = D(e_1)\hspace{4.3cm}\\
\Leftrightarrow& \left( R_{22} - \eta\right)e_1 + R_{32}e_2 + \left( R_{44} - \eta\right)e_1 = \left( R_{11} - \eta\right)e_1 + R_{21}e_2 + R_{31}e_3 + R_{41}e_4
\end{eqnarray*}
From the components of $e_i$, we obtain
$$\left\lbrace\begin{array}{lll}
R_{22} - \eta + R_{44} - \eta = R_{11} - \eta \Longrightarrow \eta = R_{22} + R_{44} - R_{11} \\
\\
R_{32} = R_{21}, \qquad R_{31} = R_{41} = 0
\end{array}\right.$$
Likewise, the condition below is satisfied
\begin{eqnarray*}
& [D(e_3), e_4] + [e_3, D(e_4)] = D[e_3, e_4] \hspace{7.1cm}\\ \Leftrightarrow & \left[ R_{13}e_1 + R_{23}e_2 + \left( R_{33} - \eta\right)e_3 + R_{43}e_4, e_4\right] +\hspace{5.4cm}\\ & \left[ e_3, R_{14}e_1 + R_{24}e_2 + R_{34}e_3 + \left( R_{44} - \eta\right)e_4\right] = D(e_2)\hspace{4.3cm}\\
\Leftrightarrow& R_{23}e_1 + \left( R_{33} - \eta\right)e_2 + \left( R_{44} - \eta\right)e_2 = R_{12}e_1 + \left( R_{22} - \eta\right)e_2 + R_{32}e_3 + R_{42}e_4 
\end{eqnarray*}
By identifying each component of $e_i$, we deduce that
$$\left\lbrace\begin{array}{lll}
R_{33} - \eta + R_{44} - \eta = R_{22} - \eta \Longrightarrow \eta = R_{33} + R_{44} - R_{22} \\
\\
R_{23} = R_{12}, \qquad R_{32} = R_{42} = 0
\end{array}\right.$$
Then the derivation $D$ becomes
$$D = \begin{bmatrix}
R_{11} - \eta & R_{12} & R_{13} & R_{14}\\
\\
0 & R_{22} - \eta & R_{12} & R_{24}\\
\\
0 & 0 & R_{33} - \eta & R_{34}\\
\\
0 & 0 & R_{43} & R_{44} - \eta
\end{bmatrix}.$$
Now, it is easy to see that the conditions $[D(e_1), e_i] + [e_1, D(e_i)] = 0$, for $i = 2, 3, 4$ are trivial. But the relation $[D(e_2), e_3] + [e_2, D(e_3)] = 0$ gives $R_{43} = 0$. Thus, we have obtained the conclusions of the theorem.
\end{proof}
\subsection{The Lie algebra $\mathfrak{g}_{4.2}^{\alpha}$}
The Lie algebra $\mathfrak{g}_{4.2}^{\alpha}$ ($\alpha \neq 0$) has a basis $\mathcal{B} = \left\lbrace e_1, e_2, e_3, e_4\right\rbrace$ such that the nonzero Lie brackets are 
$$[e_1, e_4] = \alpha e_1, \qquad [e_2, e_4] = e_2, \qquad [e_3, e_4] = e_2 + e_3.$$ 
\begin{theorem}
A pseudo-Riemannian inner product on $\mathfrak{g}_{4.2}^{\alpha}$ is an algebraic Ricci soliton if and only if its Ricci operator satisfies the conditions below
\begin{multicols}{2}
\begin{itemize}
\item $R_{32} = R_{4i} = 0, \; i = 1, 2, 3$
\item $R_{21} + R_{31} = \alpha R_{21}$
\item $R_{12} = \alpha R_{12}$
\item $R_{33} = R_{22}$
\item $R_{31} = \alpha R_{31}$
\item $R_{12} + R_{13} = \alpha R_{13}$
\end{itemize}
\end{multicols}
In this situation, the Ricci operator can be written as $\operatorname{Ric} = \eta I_4 + D$ where $\eta = R_{44}$ and
$$D = \begin{bmatrix}
R_{11} - R_{44} & R_{12} & R_{13} & R_{14}\\
\\
R_{21} & R_{22} - R_{44} & R_{23} & R_{24}\\
\\
R_{31} & 0 & R_{22} - R_{44} & R_{34}\\
\\
0 & 0 & 0 & 0
\end{bmatrix}$$
is a derivation of $\mathfrak{g}_{4.2}^{\alpha}$ with respect to the basis $\mathcal{B}$. Note that if $R_{21} = R_{31} = R_{12} = R_{13} = 0$, there is no additional condition over $\alpha$. But if $R_{31} \neq 0$ or $R_{12} \neq 0$, then $\alpha = 1$ which implies that $R_{31} = R_{12} = 0$.
\end{theorem}
\begin{proof}
Given that $D$ in (\ref{derivation}) is a derivation of $\mathfrak{g}_{4.2}^{\alpha}$, the following requirement is met
\begin{eqnarray*}
& [D(e_1), e_4] + [e_1, D(e_4)] = D[e_1, e_4] \hspace{3.4cm}\\ \Leftrightarrow & \left[ \left( R_{11} - \eta\right)e_1 + R_{21}e_2 + R_{31}e_3 + R_{41}e_4, e_4\right] + \hspace{1.7cm}\\ & \left[ e_1, R_{14}e_1 + R_{24}e_2 + R_{34}e_3 + \left(R_{44} - \eta\right)e_4\right] = \alpha D(e_1)\hspace{0.36cm}\\
\Leftrightarrow& \alpha\left(R_{11} - \eta\right)e_1 + R_{21}e_2 + R_{31}(e_2 + e_3) + \alpha\left(R_{44} - \eta\right)e_1 \\ 
& = \alpha\left(R_{11} - \eta\right)e_1 + \alpha R_{21}e_2 + \alpha R_{31}e_3 + \alpha R_{41}e_4 \hspace{1.4cm}
\end{eqnarray*}
By identifying the components of $e_i$, we obtain that
$$\left\lbrace\begin{array}{lll}
R_{11} - \eta + R_{44} - \eta = R_{11} - \eta \Longrightarrow \eta = R_{44} \\
\\
R_{21} + R_{31} = \alpha R_{21}, \qquad R_{31} = \alpha R_{31}, \qquad R_{41} = 0
\end{array}\right.$$
In a similar manner, the following requirement holds
\begin{eqnarray*}
& [D(e_2), e_4] + [e_2, D(e_4)] = D[e_2, e_4] \hspace{8.08cm}\\ \Leftrightarrow & \left[ R_{12}e_1 + \left( R_{22} - \eta\right)e_2 + R_{32}e_3 + R_{42}e_4, e_4\right] + \left[ e_2, R_{14}e_1 + R_{24}e_2 + R_{34}e_3\right] = D(e_2)\\
\Leftrightarrow& \alpha R_{12}e_1 + \left(R_{22} - \eta\right)e_2 + R_{32}(e_2 + e_3) = R_{12}e_1 + \left( R_{22} - \eta\right)e_2 + R_{32}e_3 + R_{42}e_4\hspace{0.5cm}
\end{eqnarray*}
By examining the components of $e_i$, we find that
$$\left\lbrace\begin{array}{lll}
\alpha R_{12} = R_{12}, \qquad R_{42} = 0 \\
\\
R_{22} - \eta + R_{32} = R_{22} - \eta \Longrightarrow R_{32} = 0
\end{array}\right.$$
Again, the following requirement is met
\begin{eqnarray*}
& [D(e_3), e_4] + [e_3, D(e_4)] = D[e_3, e_4] \hspace{4.2cm}\\ \Leftrightarrow & \left[ R_{13}e_1 + R_{23}e_2 + \left( R_{33} - \eta\right)e_3 + R_{43}e_4, e_4\right] +\hspace{2.5cm}\\ & \left[ e_3, R_{14}e_1 + R_{24}e_2 + R_{34}e_3\right] = D(e_2) + D(e_3)\hspace{2.5cm}\\
\Leftrightarrow& \alpha R_{13}e_1 + R_{23}e_2 + \left( R_{33} - \eta\right)(e_2 + e_3) =\hspace{3.2cm}\\ & \left(R_{12} + R_{13}\right)e_1 + \left( R_{22} - \eta + R_{23}\right)e_2 + \left( R_{33} - \eta\right)e_3 + R_{43}e_4
\end{eqnarray*}
Identifying the components of $e_i$ yields
$$\left\lbrace\begin{array}{lll}
\alpha R_{13} = R_{12} + R_{13}, \qquad R_{43} = 0 \\
\\
R_{23} + R_{33} - \eta = R_{22} - \eta + R_{23} \Longrightarrow R_{33} = R_{22}
\end{array}\right.$$
Then the derivation $D$ becomes
$$D = \begin{bmatrix}
R_{11} - R_{44} & R_{12} & R_{13} & R_{14}\\
\\
R_{21} & R_{22} - R_{44} & R_{12} & R_{24}\\
\\
R_{31} & 0 & R_{22} - R_{44} & R_{34}\\
\\
0 & 0 & 0 & 0
\end{bmatrix}.$$
At this stage, it is straightforward to check that the conditions $[D(e_1), e_i] + [e_1, D(e_i)] = 0$, for $i = 2, 3$, and $[D(e_2), e_3] + [e_2, D(e_3)] = 0$, are all trivial. The conditions over the elements $R_{12}, R_{21}, R_{13}$, and $R_{31}$ may be summarized as follows
\begin{multicols}{2}
\begin{itemize}
\item $R_{21} + R_{31} = \alpha R_{21}$
\item $R_{12} = \alpha R_{12}$
\item $R_{31} = \alpha R_{31}$
\item $R_{12} + R_{13} = \alpha R_{13}$
\end{itemize}
\end{multicols}
Hence we see that if $R_{21} = R_{31} = R_{12} = R_{13} = 0$, there is no additional condition over $\alpha$. But if $R_{31} \neq 0$ or $R_{12} \neq 0$, then $\alpha = 1$ which implies that $R_{31} = R_{12} = 0$. Hence, the results asserted by the theorem follow.
\end{proof}
\subsection{The Lie algebra $\mathfrak{g}_{4.3}$}
The Lie algebra $\mathfrak{g}_{4.3}$ has a basis $\mathcal{B} = \left\lbrace e_1, e_2, e_3, e_4\right\rbrace$ such that the nonzero Lie brackets are 
$$[e_1, e_4] = e_1, \qquad [e_3, e_4] = e_2.$$
\begin{theorem}
A pseudo-Riemannian inner product on $\mathfrak{g}_{4.3}$ is an algebraic Ricci soliton if and only if its associated Ricci operator satisfies the conditions
\begin{multicols}{2}
\begin{itemize}
\item $R_{i1} = 0, \; i = 2, 3, 4$
\item $R_{32} = R_{42} = R_{43} = 0$
\item $R_{13} = R_{12} = 0$
\item $R_{33} = R_{22}$
\end{itemize}
\end{multicols}
Given these conditions, the Ricci operator can be described as $\operatorname{Ric} = \eta I_4 + D$ where $\eta = R_{44}$ and
$$D = \begin{bmatrix}
R_{11} - R_{44} & 0 & 0 & R_{14}\\
\\
0 & R_{22} - R_{44} & R_{23} & R_{24}\\
\\
0 & 0 & R_{22} - R_{44} & R_{34}\\
\\
0 & 0 & 0 & 0
\end{bmatrix}$$
is a derivation of $\mathfrak{g}_{4.3}$ with respect to the basis $\mathcal{B}$.
\end{theorem}
\begin{proof}
As a derivation of $\mathfrak{g}_{4.3}$, $D$ in (\ref{derivation}) satisfies the following property
\begin{eqnarray*}
& [D(e_1), e_4] + [e_1, D(e_4)] = D[e_1, e_4] \hspace{7.2cm}\\ \Leftrightarrow & \left[ \left( R_{11} - \eta\right)e_1 + R_{21}e_2 + R_{31}e_3 + R_{41}e_4, e_4\right] + \hspace{5.4cm}\\ & \left[ e_1, R_{14}e_1 + R_{24}e_2 + R_{34}e_3 + \left(R_{44} - \eta\right)e_4\right] = D(e_1)\hspace{4.3cm}\\
\Leftrightarrow& \left( R_{11} - \eta\right)e_1 + R_{31}e_2 + \left( R_{44} - \eta\right)e_1 = \left( R_{11} - \eta\right)e_1 + R_{21}e_2 + R_{31}e_3 + R_{41}e_4
\end{eqnarray*}
From the components of $e_i$, it follows that
$$\left\lbrace\begin{array}{lll}
R_{11} - \eta + R_{44} - \eta = R_{11} - \eta \Longrightarrow \eta = R_{44} \\
\\
R_{31} = R_{21}, \qquad R_{31} = R_{41} = 0
\end{array}\right.$$
Also, $D$ satisfies the following property
\begin{eqnarray*}
& [D(e_3), e_4] + [e_3, D(e_4)] = D[e_3, e_4] \hspace{8.1cm}\\ \Leftrightarrow & \left[ R_{13}e_1 + R_{23}e_2 + \left( R_{33} - \eta\right)e_3 + R_{43}e_4, e_4\right] + \left[ e_3, R_{14}e_1 + R_{24}e_2 + R_{34}e_3\right] = D(e_2)\\
\Leftrightarrow& R_{13}e_1 + \left( R_{33} - \eta\right)e_2 = R_{12}e_1 + \left( R_{22} - \eta\right)e_2 + R_{32}e_3 + R_{42}e_4 \hspace{3.4cm}
\end{eqnarray*}
An analysis of the components of $e_i$ shows that
$$\left\lbrace\begin{array}{lll}
R_{13} = R_{12}, \qquad R_{32} = R_{42} = 0\\
\\
R_{33} - \eta = R_{22} - \eta \Longrightarrow R_{33} = R_{22}
\end{array}\right.$$
Then the derivation $D$ becomes
$$D = \begin{bmatrix}
R_{11} - R_{44} & R_{12} & R_{12} & R_{14}\\
\\
0 & R_{22} - R_{44} & R_{23} & R_{24}\\
\\
0 & 0 & R_{22} - R_{44} & R_{34}\\
\\
0 & 0 & R_{43} & 0
\end{bmatrix}.$$
Now, one can easily see that the condition $[D(e_1), e_2] + [e_1, D(e_2)] = 0$ is trivial. But the condition $[D(e_1), e_3] + [e_1, D(e_3)] = 0$ gives $R_{43} = 0$. In the same way, the condition $[D(e_2), e_3] + [e_2, D(e_3)] = 0$ is trivial. But the condition $[D(e_2), e_4] + [e_2, D(e_4)] = 0$ gives $R_{12} = 0$. Consequently, we have obtained all the required implications of the theorem.
\end{proof}
\subsection{The Lie algebra $\mathfrak{g}_{4.4}$}
The Lie algebra $\mathfrak{g}_{4.4}$ has a basis $\mathcal{B} = \left\lbrace e_1, e_2, e_3, e_4\right\rbrace$ such that the nonzero Lie brackets are 
$$[e_1, e_4] = e_1, \qquad [e_2, e_4] = e_1 + e_2, \qquad [e_3, e_4] = e_2 + e_3.$$
\begin{theorem}
A pseudo-Riemannian inner product on $\mathfrak{g}_{4.4}$ is an algebraic Ricci soliton if and only if its Ricci operator satisfies
\begin{multicols}{2}
\begin{itemize}
\item $R_{i1} = 0, \; i = 2, 3, 4$
\item $R_{32} = R_{42} = R_{43} = 0$
\item $R_{23} = R_{12}$
\item $R_{33} = R_{22} = R_{11}$
\end{itemize}
\end{multicols}
Under the stated conditions, the Ricci operator can be expressed as $\operatorname{Ric} = \eta I_4 + D$ where $\eta = R_{44}$ and
$$D = \begin{bmatrix}
R_{11} - R_{44} & R_{12} & R_{13} & R_{14}\\
\\
0 & R_{11} - R_{44} & R_{12} & R_{24}\\
\\
0 & 0 & R_{11} - R_{44} & R_{34}\\
\\
0 & 0 & 0 & 0
\end{bmatrix}$$
is a derivation of $\mathfrak{g}_{4.4}$ with respect to the basis $\mathcal{B}$.
\end{theorem}
\begin{proof}
Because $D$ in (\ref{derivation}) is a derivation of $\mathfrak{g}_{4.4}$, we have
\begin{eqnarray*}
& [D(e_1), e_4] + [e_1, D(e_4)] = D[e_1, e_4] \hspace{3.8cm}\\ \Leftrightarrow & \left[ \left( R_{11} - \eta\right)e_1 + R_{21}e_2 + R_{31}e_3 + R_{41}e_4, e_4\right] + \hspace{2.1cm}\\ & \left[ e_1, R_{14}e_1 + R_{24}e_2 + R_{34}e_3 + \left(R_{44} - \eta\right)e_4\right] = D(e_1)\hspace{1cm}\\
\Leftrightarrow& \left( R_{11} - \eta\right)e_1 + R_{21}(e_1 + e_2) + R_{31}(e_2 + e_3) + \left( R_{44} - \eta\right)e_1\\
& = \left( R_{11} - \eta\right)e_1 + R_{21}e_2 + R_{31}e_3 + R_{41}e_4\hspace{3cm}
\end{eqnarray*}
Upon identifying the components of $e_i$, we conclude that
$$\left\lbrace\begin{array}{lll}
R_{11} - \eta + R_{21} + R_{44} - \eta = R_{11} - \eta \Longrightarrow \eta = R_{21} + R_{44} \\
\\
R_{21} + R_{31} = R_{21} \Longrightarrow R_{31} = 0, \qquad R_{41} = 0
\end{array}\right.$$
By the same argument, we have
\begin{eqnarray*}
& [D(e_2), e_4] + [e_2, D(e_4)] = D[e_2, e_4]\hspace{5cm}\\ \Leftrightarrow & \left[ R_{12}e_1 + \left( R_{22} - \eta\right)e_2 + R_{32}e_3 + R_{42}e_4, e_4\right] +\hspace{3.3cm}\\
& \left[ e_2, R_{14}e_1 + R_{24}e_2 + R_{34}e_3 + \left( R_{44} - \eta\right)e_4\right] = D(e_1) + D(e_2)\hspace{0.7cm}\\
\Leftrightarrow& R_{12}e_1 + \left( R_{22} - \eta\right)(e_1 + e_2) + R_{32}(e_2 + e_3) + \left( R_{44} - \eta\right)(e_1 + e_2)\\
& = \left(R_{11} - \eta + R_{12}\right)e_1 + \left(R_{21} + R_{22} - \eta \right)e_2 + R_{32}e_3 + R_{42}e_4\hspace{1cm}
\end{eqnarray*}
From the components of $e_i$, we obtain
$$\left\lbrace\begin{array}{lll}
R_{12} + R_{22} - \eta  + R_{44} - \eta = R_{11} - \eta + R_{12} \Longrightarrow \eta = R_{22} + R_{44} - R_{11} \\
\\
R_{22} - \eta + R_{32}  + R_{44} - \eta = R_{21} + R_{22} - \eta \Longrightarrow \eta = R_{32} + R_{44} - R_{21} \\
\\
R_{42} = 0
\end{array}\right.$$
Proceeding analogously, we have
\begin{eqnarray*}
& [D(e_3), e_4] + [e_3, D(e_4)] = D[e_3, e_4]\hspace{5.8cm}\\ \Leftrightarrow & \left[ R_{13}e_1 + R_{23}e_2 + \left( R_{33} - \eta\right)e_3 + R_{43}e_4, e_4\right] +\hspace{4.05cm}\\
& \left[ e_3, R_{14}e_1 + R_{24}e_2 + R_{34}e_3 + \left( R_{44} - \eta\right)e_4\right] = D(e_2) + D(e_3)\hspace{1.4cm}\\
\Leftrightarrow& R_{13}e_1 + R_{23}(e_1 + e_2) + \left( R_{33} - \eta\right)(e_2 + e_3) + \left( R_{44} - \eta\right)(e_2 + e_3)\hspace{0.7cm}\\
& = \left(R_{12} + R_{13}\right)e_1 + \left(R_{22} - \eta + R_{23} \right)e_2 + \left(R_{32} + R_{33} - \eta\right)e_3 + R_{43}e_4
\end{eqnarray*}
By identifying each component of $e_i$, we deduce that
$$\left\lbrace\begin{array}{lll}
R_{13} + R_{23} = R_{12} + R_{13} \Longrightarrow R_{23} = R_{12}\\
\\
R_{23} + R_{33} - \eta  + R_{44} - \eta = R_{22} - \eta + R_{23} \Longrightarrow \eta = R_{33} + R_{44} - R_{22} \\
\\
R_{33} - \eta  + R_{44} - \eta = R_{32} + R_{33} - \eta \Longrightarrow \eta = R_{44} - R_{32}\\
\\
R_{43} = 0
\end{array}\right.$$
From the above three bracket conditions, we remark that
\begin{align}
\eta &= R_{32} + R_{44} - R_{21} \label{equ1} \\
\eta &= R_{44} - R_{32} \label{equ2} \\
\eta &= R_{44} + R_{21} \label{equ3} 
\end{align}
We see that $(\ref{equ1}) + (\ref{equ2}) \Rightarrow 2\eta = 2R_{44} - R_{21}$. On the other hand, multiplying the equation (\ref{equ3}) by $2$ reveals that $2\eta = 2R_{44} + 2R_{21}$. Hence
$$2R_{44} + 2R_{21} = 2R_{44} - R_{21} \Rightarrow 2R_{21} = - R_{21} \Rightarrow  R_{21} = 0.$$
Therefore $\eta = R_{44} = R_{44} + R_{32} \Rightarrow R_{32} = 0$. According to this fact and by seeing the above three bracket conditions again, we obtain that
$$\eta = R_{44} = R_{22} + R_{44} - R_{11} = R_{33} + R_{44} - R_{22} \Longrightarrow R_{33} = R_{22} = R_{11}.$$
Thus, the derivation $D$ becomes
$$D = \begin{bmatrix}
R_{11} - R_{44} & R_{12} & R_{13} & R_{14}\\
\\
0 & R_{11} - R_{44} & R_{12} & R_{24}\\
\\
0 & 0 & R_{11} - R_{44} & R_{34}\\
\\
0 & 0 & 0 & 0
\end{bmatrix}.$$
At this stage, it is straightforward to check that the conditions $[D(e_1), e_i] + [e_1, D(e_i)] = 0$, for $i = 2, 3$, and $[D(e_2), e_3] + [e_2, D(e_3)] = 0$, are all trivial. Thus, the theorem's conclusions follow.
\end{proof}
\subsection{The Lie algebra $\mathfrak{g}_{4.5}^{\alpha, \beta}$}
The Lie algebra $\mathfrak{g}_{4.5}^{\alpha, \beta}$ ($-1 < \alpha \leq \beta \leq1, \alpha\beta \neq 0$ or $\alpha = -1, 0 < \beta \leq 1$) has a basis $\mathcal{B} = \left\lbrace e_1, e_2, e_3, e_4\right\rbrace$ such that the nonzero Lie brackets are $$[e_1, e_4] = e_1, \qquad [e_2, e_4] = \beta e_2, \qquad [e_3, e_4] = \alpha e_3.$$
\begin{theorem}
A pseudo-Riemannian inner product on $\mathfrak{g}_{4.5}^{\alpha, \beta}$ is an algebraic Ricci soliton if and only if its Ricci operator satisfies the following identities
\begin{multicols}{2}
\begin{itemize}
\item $R_{4i} = 0, \; i = 1, 2, 3$
\item $R_{21} = \beta R_{21}$
\item $R_{31} = \alpha R_{31}$
\item $R_{12} = \beta R_{12}$
\item $\alpha R_{32} = \beta R_{32}$
\item $R_{13} = \alpha R_{13}$
\item $\beta R_{23} = \alpha R_{23}$
\end{itemize}
\end{multicols}
Provided these identities are satisfied, the Ricci operator can take the form\\ $\operatorname{Ric} = \eta I_4 + D$ where $\eta = R_{44}$ and
$$D = \begin{bmatrix}
R_{11} - R_{44} & R_{12} & R_{13} & R_{14}\\
\\
R_{21} & R_{22} - R_{44} & R_{23} & R_{24}\\
\\
R_{31} & R_{32} & R_{33} - R_{44} & R_{34}\\
\\
0 & 0 & 0 & 0
\end{bmatrix}$$
is a derivation of $\mathfrak{g}_{4.5}^{\alpha, \beta}$ with respect to the basis $\mathcal{B}$. Note that if
$$R_{21} = R_{31} = R_{12} = R_{13} = R_{23} = R_{32} = 0,$$ then there are no additional conditions over $\alpha$ and $\beta$. But we see that
$$\left\lbrace\begin{array}{lll}
R_{21} \neq 0 \Longrightarrow \beta = 1\\
R_{31} \neq 0 \Longrightarrow \alpha = 1\\
R_{12} \neq 0 \Longrightarrow \beta = 1\\
R_{32} \neq 0 \Longrightarrow \alpha =  \beta\\
R_{13} \neq 0 \Longrightarrow \alpha = 1\\
R_{23} \neq 0 \Longrightarrow \alpha = \beta
\end{array}\right.$$
\end{theorem}
\begin{proof}
Being a derivation of $\mathfrak{g}_{4.5}^{\alpha, \beta}$, $D$ in (\ref{derivation}) satisfies the following rule
\begin{eqnarray*}
& [D(e_1), e_4] + [e_1, D(e_4)] = D[e_1, e_4] \hspace{2.8cm}\\ \Leftrightarrow & \left[ \left( R_{11} - \eta\right)e_1 + R_{21}e_2 + R_{31}e_3 + R_{41}e_4, e_4\right] + \hspace{1.1cm}\\ & \left[ e_1, R_{14}e_1 + R_{24}e_2 + R_{34}e_3 + \left(R_{44} - \eta\right)e_4\right] = D(e_1)\\
\Leftrightarrow& \left( R_{11} - \eta\right)e_1 + \beta R_{21}e_2 + \alpha R_{31}e_3 + \left( R_{44} - \eta\right)e_1\hspace{0.7cm}\\
& = \left( R_{11} - \eta\right)e_1 + R_{21}e_2 + R_{31}e_3 + R_{41}e_4\hspace{1.9cm}
\end{eqnarray*}
By identifying the components of $e_i$, we obtain that
$$\left\lbrace\begin{array}{lll}
R_{11} - \eta + R_{44} - \eta = R_{11} - \eta \Longrightarrow \eta = R_{44} \\
\\
\beta R_{21} = R_{21}, \quad \alpha R_{31} = R_{31}, \quad R_{41} = 0
\end{array}\right.$$
In the same way, $D$ satisfies the following rule
\begin{eqnarray*}
& [D(e_2), e_4] + [e_2, D(e_4)] = D[e_2, e_4]\hspace{3.08cm}\\ \Leftrightarrow & \left[ R_{12}e_1 + \left( R_{22} - \eta\right)e_2 + R_{32}e_3 + R_{42}e_4, e_4\right] +\hspace{1.38cm}\\
& \left[ e_2, R_{14}e_1 + R_{24}e_2 + R_{34}e_3 + \left( R_{44} - \eta\right)e_4\right] = \beta D(e_2)\\
\Leftrightarrow& R_{12}e_1 + \beta\left( R_{22} - \eta\right)e_2 + \alpha R_{32}e_3 + \beta\left( R_{44} - \eta\right)e_2\hspace{0.5cm}\\
& = \beta R_{12}e_1 + \beta\left(R_{22} - \eta \right)e_2 + \beta R_{32}e_3 + \beta R_{42}e_4\hspace{1cm}
\end{eqnarray*}
By examining the components of $e_i$, we find that
$$\left\lbrace\begin{array}{lll}
R_{12} = \beta R_{12}, \quad \alpha R_{32} = \beta R_{32}, \quad \beta R_{42} = 0 \overset{\beta \neq 0}{\Longrightarrow} R_{42} = 0\\
\\
R_{22} - \eta + R_{44} - \eta = R_{22} - \eta \Longrightarrow \eta = R_{44} 
\end{array}\right.$$
As well, the rule below holds
\begin{eqnarray*}
& [D(e_3), e_4] + [e_3, D(e_4)] = D[e_3, e_4]\hspace{3.08cm}\\ \Leftrightarrow & \left[ R_{13}e_1 + R_{23}e_2 + \left( R_{33} - \eta\right)e_3 + R_{43}e_4, e_4\right] +\hspace{1.35cm}\\
& \left[ e_3, R_{14}e_1 + R_{24}e_2 + R_{34}e_3 + \left( R_{44} - \eta\right)e_4\right] = \alpha D(e_3)\\
\Leftrightarrow& R_{13}e_1 + \beta R_{23}e_2 + \alpha\left( R_{33} - \eta\right)e_3 + \alpha\left( R_{44} - \eta\right)e_3\hspace{0.5cm}\\
& = \alpha R_{13}e_1 + \alpha R_{23}e_2 + \alpha\left(R_{33} - \eta\right)e_3 + \alpha R_{43}e_4\hspace{1cm}
\end{eqnarray*}
Identifying the components of $e_i$ yields
$$\left\lbrace\begin{array}{lll}
R_{13} = \alpha R_{13}, \quad \beta R_{23} = \alpha R_{23}, \quad R_{43} = 0\\
\\
R_{33} - \eta  + R_{44} - \eta = R_{33} - \eta \Longrightarrow \eta = R_{44}
\end{array}\right.$$
Therefore, the derivation $D$ takes the form
$$D = \begin{bmatrix}
R_{11} - R_{44} & R_{12} & R_{13} & R_{14}\\
\\
R_{21} & R_{22} - R_{44} & R_{23} & R_{24}\\
\\
R_{31} & R_{32} & R_{33} - R_{44} & R_{34}\\
\\
0 & 0 & 0 & 0
\end{bmatrix}.$$
Now, all the remaining vanishing brackets are trivially satisfied by $D$. We can see that, if
$$R_{21} = R_{31} = R_{12} = R_{13} = R_{23} = R_{32} = 0,$$ then there are no additional conditions over $\alpha$ and $\beta$. But
$$\left\lbrace\begin{array}{lll}
R_{21} \neq 0 \Longrightarrow \beta = 1\\
R_{31} \neq 0 \Longrightarrow \alpha = 1\\
R_{12} \neq 0 \Longrightarrow \beta = 1\\
R_{32} \neq 0 \Longrightarrow \alpha =  \beta\\
R_{13} \neq 0 \Longrightarrow \alpha = 1\\
R_{23} \neq 0 \Longrightarrow \alpha = \beta
\end{array}\right.$$
\end{proof}
\subsection{The Lie algebra $\mathfrak{g}_{4.6}^{\alpha, \beta}$}
The Lie algebra $\mathfrak{g}_{4.6}^{\alpha, \beta}$ ($\alpha > 0$ and $\beta \in \mathbb{R}$) has a basis $\mathcal{B} = \left\lbrace e_1, e_2, e_3, e_4\right\rbrace$ such that the nonzero Lie brackets are $$[e_1, e_4] = \alpha e_1, \qquad [e_2, e_4] = \beta e_2 - e_3, \qquad [e_3, e_4] = e_2 + \beta e_3.$$
\begin{theorem}
A pseudo-Riemannian inner product on $\mathfrak{g}_{4.6}^{\alpha, \beta}$ is an algebraic Ricci soliton if and only if its Ricci operator fulfills the following relations
\begin{multicols}{2}
\begin{itemize}
\item $R_{4i} = 0, \; i = 1, 2, 3$
\item $R_{33} = R_{22}$
\item $R_{32} = -R_{23}$
\item $\beta R_{21} + R_{31} = \alpha R_{21}$
\item $-R_{21} + \beta R_{31} = \alpha R_{31}$
\item $\alpha R_{12} = \beta R_{12} - R_{13}$
\item $\alpha R_{13} = R_{12} + \beta R_{13}$
\end{itemize}
\end{multicols}
In this case, the Ricci operator can be stated as $\operatorname{Ric} = \eta I_4 + D$ where $\eta = R_{44}$ and 
$$D = \begin{bmatrix}
R_{11} - R_{44} & R_{12} & R_{13} & R_{14}\\
\\
R_{21} & R_{22} - R_{44} & R_{23} & R_{24}\\
\\
R_{31} & -R_{23} & R_{22} - R_{44} & R_{34}\\
\\
0 & 0 & 0 & 0
\end{bmatrix}$$
is a derivation of $\mathfrak{g}_{4.6}^{\alpha, \beta}$ with respect to the basis $\mathcal{B}$. In particular, if $\beta = 0$, then we see that $\alpha R_{12} = \alpha^2R_{13} = -R_{13}$. This implies that $R_{13} = R_{12} = 0$. Similarly, we remark that $\alpha R_{31} = \alpha^2R_{21} = -R_{21}$, hence $R_{21} = R_{31} = 0$. Therefore, the derivation $D$ becomes
$$D = \begin{bmatrix}
R_{11} - R_{44} & 0 & 0 & R_{14}\\
\\
0 & R_{22} - R_{44} & R_{23} & R_{24}\\
\\
0 & -R_{23} & R_{22} - R_{44} & R_{34}\\
\\
0 & 0 & 0 & 0
\end{bmatrix}.$$
\end{theorem}
\begin{proof}
Given that $D$ in (\ref{derivation}) is a derivation of $\mathfrak{g}_{4.6}^{\alpha, \beta}$, it satisfies the following condition
\begin{eqnarray*}
& [D(e_1), e_4] + [e_1, D(e_4)] = D[e_1, e_4] \hspace{5.1cm}\\ \Leftrightarrow & \left[ \left( R_{11} - \eta\right)e_1 + R_{21}e_2 + R_{31}e_3 + R_{41}e_4, e_4\right] + \hspace{3.35cm}\\ & \left[ e_1, R_{14}e_1 + R_{24}e_2 + R_{34}e_3 + \left(R_{44} - \eta\right)e_4\right] = \alpha D(e_1)\hspace{2cm}\\
\Leftrightarrow& \alpha\left( R_{11} - \eta\right)e_1 + R_{21}(\beta e_2 - e_3) + R_{31}(e_2 + \beta e_3) + \alpha\left( R_{44} - \eta\right)e_1\\
& = \alpha\left( R_{11} - \eta\right)e_1 + \alpha R_{21}e_2 + \alpha R_{31}e_3 + \alpha R_{41}e_4\hspace{3.2cm}
\end{eqnarray*}
From the components of $e_i$, it follows that
$$\left\lbrace\begin{array}{lll}
R_{11} - \eta + R_{44} - \eta = R_{11} - \eta \Longrightarrow \eta = R_{44} \\
\\
\beta R_{21} + R_{31} = \alpha R_{21}\\
\\
-R_{21} + \beta R_{31} = \alpha R_{31}, \qquad R_{41} = 0
\end{array}\right.$$
Similarly, $D$ satisfies the following condition
\begin{eqnarray*}
& [D(e_2), e_4] + [e_2, D(e_4)] = D[e_2, e_4]\hspace{8.88cm}\\ \Leftrightarrow & \left[ R_{12}e_1 + \left( R_{22} - \eta\right)e_2 + R_{32}e_3 + R_{42}e_4, e_4\right] +\hspace{7.15cm}\\
& \left[ e_2, R_{14}e_1 + R_{24}e_2 + R_{34}e_3 + \left( R_{44} - \eta\right)e_4\right] = \beta D(e_2) - D(e_3)\hspace{4.3cm}\\
\Leftrightarrow& \alpha R_{12}e_1 + \left( R_{22} - \eta\right)(\beta e_2 - e_3) + R_{32}(e_2 + \beta e_3) + \left( R_{44} - \eta\right)(\beta e_2 - e_3) =\hspace{2.2cm}\\
& \left(\beta R_{12} - R_{13}\right)e_1 + \left(\beta\left(R_{22} - \eta \right) - R_{23}\right)e_2 + \left(\beta R_{32} - \left(R_{33} - \eta \right)\right)e_3 + \left( \beta R_{42} - R_{43}\right)e_4
\end{eqnarray*}
An analysis of the components of $e_i$ shows that
$$\left\lbrace\begin{array}{lll}
\alpha R_{12} = \beta R_{12} - R_{13}\\
\\
\beta\left(R_{22} - \eta\right) + R_{32} + \beta\left(R_{44} - \eta\right) = \beta\left(R_{22} - \eta\right) - R_{23} \overset{\eta = R_{44}}{\Longrightarrow} R_{32} = -R_{23}\\
\\
\eta - R_{22} + \beta R_{32} + \eta - R_{44} = \beta R_{32} + \eta - R_{33} \overset{\eta = R_{44}}{\Longrightarrow} R_{33} = R_{22}\\
\\
\beta R_{42} = R_{43} 
\end{array}\right.$$
Likewise, the following condition is satisfied
\begin{eqnarray*}
& [D(e_3), e_4] + [e_3, D(e_4)] = D[e_3, e_4]\hspace{8.4cm}\\ \Leftrightarrow & \left[ R_{13}e_1 + R_{23}e_2 + \left( R_{33} - \eta\right)e_3 + R_{43}e_4, e_4\right] +\hspace{6.7cm}\\
& \left[ e_3, R_{14}e_1 + R_{24}e_2 + R_{34}e_3 + \left( R_{44} - \eta\right)e_4\right] = D(e_2) + \beta D(e_3)\hspace{3.8cm}\\
\Leftrightarrow& \alpha R_{13}e_1 + R_{23}(\beta e_2 - e_3) + \left( R_{33} - \eta\right)(e_2 + \beta e_3) + \left( R_{44} - \eta\right)(e_2 + \beta e_3) =\hspace{1.8cm}\\
& \left(R_{12} + \beta R_{13}\right)e_1 + \left(R_{22} - \eta + \beta R_{23}\right)e_2 + \left(R_{32} + \beta\left(R_{33} - \eta \right)\right)e_3 + \left(R_{42} + \beta R_{43}\right)e_4
\end{eqnarray*}
Upon identifying the components of $e_i$, we conclude that
$$\left\lbrace\begin{array}{lll}
\alpha R_{13} = R_{12} + \beta R_{13}\\
\\
\beta R_{23} + R_{33} - \eta + R_{44} - \eta = R_{22} - \eta + \beta R_{23} \overset{\eta = R_{44}}{\Longrightarrow} R_{33} = R_{22}\\
\\
-R_{23} + \beta\left(R_{33} - \eta\right) + \beta\left(R_{44} - \eta\right) = R_{32} + \beta\left(R_{33} - \eta\right) \overset{\eta = R_{44}}{\Longrightarrow} R_{32} = -R_{23}\\
\\
R_{42} + \beta R_{43} = 0
\end{array}\right.$$
Next, let us analyse the condition $[D(e_1), e_2] + [e_1, D(e_2)] = 0$. Since $R_{41} = 0$, then $[D(e_1), e_2]$ is zero. Hence 
$$[e_1, D(e_2)] = 0 \Leftrightarrow [e_1, R_{12}e_1 + (R_{22} - \eta)e_2 + R_{32}e_3 + R_{42}e_4] = 0 \Leftrightarrow \alpha R_{42}e_1 = 0 \Leftrightarrow R_{42} = 0.$$
In the same way, the property $[D(e_1), e_3] + [e_1, D(e_3)] = 0$ shows that $R_{43} = 0$. But the property $[D(e_2), e_3] + [e_2, D(e_3)] = 0$ is trivial.
We have therefore established the theorem's stated results.
\end{proof}
\subsection{The Lie algebra $\mathfrak{g}_{4.7}$}
The Lie algebra $\mathfrak{g}_{4.7}$ has a basis $\mathcal{B} = \left\lbrace e_1, e_2, e_3, e_4\right\rbrace$ such that the nonzero Lie brackets are
$$[e_1, e_4] = 2e_1, \qquad [e_2, e_4] = e_2, \qquad [e_3, e_4] = e_2 + e_3, \qquad [e_2, e_3] = e_1.$$
\begin{theorem}
A pseudo-Riemannian inner product on $\mathfrak{g}_{4.7}$ is an algebraic Ricci soliton if and only if its Ricci operator satisfies the conditions below
\begin{multicols}{2}
\begin{itemize}
\item $R_{i1} = 0, \; i = 2, 3, 4$
\item $R_{32} = R_{42} = R_{43} = 0$
\item $R_{34} = -R_{12}$
\item $R_{24} = R_{13} - R_{12}$
\item $R_{33} = R_{22}$
\item $R_{44} = R_{22} + R_{33} - R_{11}$
\end{itemize}
\end{multicols}
In this situation, the Ricci operator can be written as $\operatorname{Ric} = \eta I_4 + D$ where\\ $\eta = R_{44} = R_{22} + R_{33} - R_{11}$ and
$$D = \begin{bmatrix}
R_{11} - R_{44} & R_{12} & R_{13} & R_{14}\\
\\
0 & R_{22} - R_{44} & R_{23} & R_{13} - R_{12}\\
\\
0 & 0 & R_{22} - R_{44} & -R_{12}\\
\\
0 & 0 & 0 & 0
\end{bmatrix}$$
is a derivation of $\mathfrak{g}_{4.7}$ with respect to the basis $\mathcal{B}$.
\end{theorem}
\begin{proof}
Since $D$ in (\ref{derivation}) is a derivation of $\mathfrak{g}_{4.7}$, the condition below holds
\begin{eqnarray*}
& [D(e_1), e_4] + [e_1, D(e_4)] = D[e_1, e_4] \hspace{3.3cm}\\ \Leftrightarrow & \left[ \left( R_{11} - \eta\right)e_1 + R_{21}e_2 + R_{31}e_3 + R_{41}e_4, e_4\right] + \hspace{1.6cm}\\ & \left[ e_1, R_{14}e_1 + R_{24}e_2 + R_{34}e_3 + \left(R_{44} - \eta\right)e_4\right] = 2D(e_1)\hspace{0.3cm}\\
\Leftrightarrow& 2\left( R_{11} - \eta\right)e_1 + R_{21}e_2 + R_{31}(e_2 + e_3) + 2\left( R_{44} - \eta\right)e_1\\
& = 2\left( R_{11} - \eta\right)e_1 + 2R_{21}e_2 + 2R_{31}e_3 + 2R_{41}e_4\hspace{1.5cm}
\end{eqnarray*}
From the components of $e_i$, we obtain
$$\left\lbrace\begin{array}{lll}
R_{11} - \eta + R_{44} - \eta = R_{11} - \eta \Longrightarrow \eta = R_{44} \\
\\
R_{21} + R_{31} = 2R_{21} \Longrightarrow R_{31} = R_{21}\\
\\
R_{31} = 2R_{31} \Longrightarrow R_{31} = 0, \qquad R_{41} = 0
\end{array}\right.$$
Likewise, the condition below is satisfied
\begin{eqnarray*}
& [D(e_2), e_4] + [e_2, D(e_4)] = D[e_2, e_4] \hspace{4.55cm}\\ \Leftrightarrow & \left[ R_{12}e_1 + \left( R_{22} - \eta\right)e_2 + R_{32}e_3 + R_{42}e_4, e_4\right] + \hspace{2.8cm}\\ & \left[ e_2, R_{14}e_1 + R_{24}e_2 + R_{34}e_3 + \left(R_{44} - \eta\right)e_4\right] = D(e_2)\hspace{1.7cm}\\
\Leftrightarrow& 2R_{12}e_1 + \left( R_{22} - \eta\right)e_2 + R_{32}(e_2 + e_3) + R_{34}e_1 + \left( R_{44} - \eta\right)e_2\\
& = R_{12}e_1 + \left( R_{22} - \eta\right)e_2 + R_{32}e_3 + R_{42}e_4\hspace{3.5cm}
\end{eqnarray*}
By identifying each component of $e_i$, we deduce that
$$\left\lbrace\begin{array}{lll}
2R_{12} + R_{34} = R_{12} \Longrightarrow R_{34} = -R_{12}\\
\\
R_{22} - \eta + R_{32} + R_{44} - \eta = R_{22} - \eta \overset{\eta = R_{44}}{\Longrightarrow} R_{32} = 0 \\
\\
R_{42} = 0
\end{array}\right.$$
Then the derivation $D$ becomes 
$$D = \begin{bmatrix}
R_{11} - R_{44} & R_{12} & R_{13} & R_{14}\\
\\
0 & R_{22} - R_{44} & R_{23} & R_{24}\\
\\
0 & 0 & R_{33} - R_{44} & -R_{12}\\
\\
0 & 0 & R_{43} & 0
\end{bmatrix}.$$ 
As well, we have
\begin{eqnarray*}
& [D(e_3), e_4] + [e_3, D(e_4)] = D[e_3, e_4] \hspace{5.45cm}\\ \Leftrightarrow & \left[ R_{13}e_1 + R_{23}e_2 + \left( R_{33} - R_{44}\right)e_3 + R_{43}e_4, e_4\right] + \hspace{3.3cm}\\ & \left[ e_3, R_{14}e_1 + R_{24}e_2 - R_{12}e_3 \right] = D(e_2) + D(e_3)\hspace{3.7cm}\\
\Leftrightarrow& 2R_{13}e_1 + R_{23}e_2 + \left( R_{33} - R_{44}\right)(e_2 + e_3) - R_{24}e_1\hspace{3.2cm}\\
& = \left( R_{12} + R_{13}\right)e_1 + \left( R_{22} - R_{44} + R_{23}\right)e_2 + \left( R_{33} - R_{44}\right)e_3 + R_{43}e_4
\end{eqnarray*}
By identifying the components of $e_i$, we obtain that
$$\left\lbrace\begin{array}{lll}
2R_{13} - R_{24} = R_{12} + R_{13} \Longrightarrow R_{24} = R_{13} - R_{12}\\
\\
R_{23} + R_{33} - R_{44} = R_{22} - R_{44} + R_{23} \Longrightarrow R_{33} = R_{22} \\
\\
R_{43} = 0
\end{array}\right.$$
Once again, the following condition is satisfied by $D$:
\begin{eqnarray*}
& [D(e_2), e_3] + [e_2, D(e_3)] = D[e_2, e_3] \hspace{7cm}\\ \Leftrightarrow & \left[ R_{12}e_1 + \left( R_{22} - R_{44}\right)e_2, e_3\right] + \left[ e_2, R_{13}e_1 + R_{23}e_2 + \left( R_{33} - R_{44}\right)e_3 \right] = D(e_1)\\
\Leftrightarrow& \left( R_{22} - R_{44}\right)e_1 + \left( R_{33} - R_{44}\right)e_1 = \left( R_{11} - R_{44}\right)e_1\hspace{4.5cm}
\end{eqnarray*}
Hence $R_{22} - R_{44} + R_{33} - R_{44} = R_{11} - R_{44} \Longrightarrow R_{44} = R_{22} + R_{33} - R_{11}$. Now, it is straightforward to see that the following remaining two conditions are trivial: 
$$[D(e_1), e_2] + [e_1, D(e_2)] = 0, \qquad [D(e_1), e_3] + [e_1, D(e_3)] = 0.$$
It follows that the consequences of the theorem have been obtained.
\end{proof}
\subsection{The Lie algebra $\mathfrak{g}_{4.8}^{-1}$}
The Lie algebra $\mathfrak{g}_{4.8}^{-1}$ has a basis $\mathcal{B} = \left\lbrace e_1, e_2, e_3, e_4\right\rbrace$ such that the nonzero Lie brackets are 
$$[e_2, e_3] = e_1, \qquad [e_2, e_4] = e_2, \qquad [e_3, e_4] = -e_3.$$
\begin{theorem}
A pseudo-Riemannian inner product on $\mathfrak{g}_{4.8}^{-1}$ is an algebraic Ricci soliton if and only if its associated Ricci operator satisfies
\begin{multicols}{2}
\begin{itemize}
\item $R_{i1} = 0, \; i = 2, 3, 4$
\item $R_{32} = R_{23} = R_{42} = R_{43} = 0$
\item $R_{44} = R_{22} + R_{33} - R_{11}$
\end{itemize}
\end{multicols}
With these conditions in place, the Ricci operator may be written as $\operatorname{Ric} = \eta I_4 + D$ where $\eta = R_{44} = R_{22} + R_{33} - R_{11}$ and
$$D = \begin{bmatrix}
R_{11} - R_{44} & R_{12} & R_{13} & R_{14}\\
\\
0 & R_{22} - R_{44} & 0 & R_{13}\\
\\
0 & 0 & R_{33} - R_{44} & R_{12}\\
\\
0 & 0 & 0 & 0
\end{bmatrix}$$
is a derivation of $\mathfrak{g}_{4.8}^{-1}$ with respect to the basis $\mathcal{B}$.
\end{theorem}
\begin{proof}
Given that $D$ in (\ref{derivation}) is a derivation of $\mathfrak{g}_{4.8}^{-1}$, the following requirement is met
\begin{eqnarray*}
& [D(e_2), e_3] + [e_2, D(e_3)] = D[e_2, e_3] \hspace{8.7cm}\\ \Leftrightarrow & \left[ R_{12}e_1 + \left( R_{22} - \eta\right)e_2 + R_{32}e_3 + R_{42}e_4, e_3\right] + \hspace{6.95cm}\\ & \left[ e_2, R_{13}e_1 + R_{23}e_2 + \left(R_{33} - \eta\right)e_3 + R_{43}e_4\right] = D(e_1)\hspace{5.85cm}\\
\Leftrightarrow& \left( R_{22} - \eta\right)e_1 + R_{42}e_3 + \left( R_{33} - \eta\right)e_1 + R_{43}e_2 = \left( R_{11} - \eta\right)e_1 + R_{21}e_2 + R_{31}e_3 + R_{41}e_4
\end{eqnarray*}
By examining the components of $e_i$, we find that
$$\left\lbrace\begin{array}{lll}
R_{22} - \eta + R_{33} - \eta = R_{11} - \eta \Longrightarrow \eta = R_{22} + R_{33} - R_{11} \\
\\
R_{43} = R_{21}, \qquad R_{42} = R_{31}, \qquad R_{41} = 0
\end{array}\right.$$
In a similar manner, the following requirement holds
\begin{eqnarray*}
& [D(e_2), e_4] + [e_2, D(e_4)] = D[e_2, e_4] \hspace{9.7cm}\\ \Leftrightarrow & \left[ R_{12}e_1 + \left( R_{22} - \eta\right)e_2 + R_{32}e_3 + R_{42}e_4, e_4\right] + \hspace{7.98cm}\\ & \left[ e_2, R_{14}e_1 + R_{24}e_2 + R_{34}e_3 + \left(R_{44} - \eta\right)e_4\right] = D(e_2)\hspace{6.88cm}\\
\Leftrightarrow& \left( R_{22} - \eta\right)e_2 - R_{32}e_3 + R_{34}e_1 + \left( R_{44} - \eta\right)e_2 = R_{12}e_1 + \left( R_{22} - \eta\right)e_2 + R_{32}e_3 + R_{42}e_4\hspace{1cm}
\end{eqnarray*}
Identifying the components of $e_i$ yields
$$\left\lbrace\begin{array}{lll}
R_{34} = R_{12}\\
\\
R_{22} - \eta + R_{44} - \eta = R_{22} - \eta \Longrightarrow \eta = R_{44}\\
\\
-R_{32} = R_{32} \Longrightarrow R_{32} = 0, \qquad R_{42} = 0
\end{array}\right.$$
Again, the following requirement is met
\begin{eqnarray*}
& [D(e_3), e_4] + [e_3, D(e_4)] = D[e_3, e_4]\hspace{3.15cm} \\ \Leftrightarrow & \left[ R_{13}e_1 + R_{23}e_2 + \left( R_{33} - \eta\right)e_3 + R_{43}e_4, e_4\right] +\hspace{1.43cm}\\ & \left[ e_3, R_{14}e_1 + R_{24}e_2 + R_{34}e_3 + \left(R_{44} - \eta\right)e_4\right] = -D(e_3)\\
\Leftrightarrow& R_{23}e_2 - \left( R_{33} - \eta\right)e_3 - R_{24}e_1 - \left( R_{44} - \eta\right)e_3\hspace{1.5cm}\\
& = -R_{13}e_1 - R_{23}e_2 - \left( R_{33} - \eta\right)e_3 - R_{43}e_4\hspace{1.8cm}
\end{eqnarray*}
From the components of $e_i$, it follows that
$$\left\lbrace\begin{array}{lll}
R_{24} = R_{13}, \qquad R_{23} = -R_{23} \Longrightarrow R_{23} = 0\\
\\
R_{33} - \eta + R_{44} - \eta = R_{33} - \eta \Longrightarrow \eta = R_{44}\\
\\
R_{43} = 0
\end{array}\right.$$
According to the above three conditions, the derivation $D$ becomes
$$D = \begin{bmatrix}
R_{11} - R_{44} & R_{12} & R_{13} & R_{14}\\
\\
0 & R_{22} - R_{44} & 0 & R_{13}\\
\\
0 & 0 & R_{33} - R_{44} & R_{12}\\
\\
0 & 0 & 0 & 0
\end{bmatrix}.$$
A straightforward verification shows that the conditions $[D(e_1), e_i] + [e_1, D(e_i)] = 0$ for $i = 2, 3, 4$ are trivially satisfied by $D$. Thus, the theorem's conclusions follow.
\end{proof}
\subsection{The Lie algebra $\mathfrak{g}_{4.8}^{\alpha}$}
The Lie algebra $\mathfrak{g}_{4.8}^{\alpha}$ ($-1 < \alpha \leq 1$) has a basis $\mathcal{B} = \left\lbrace e_1, e_2, e_3, e_4\right\rbrace$ such that the nonzero Lie brackets are
$$[e_1, e_4] = (1 + \alpha)e_1, \qquad [e_2, e_4] = e_2, \qquad [e_3, e_4] = \alpha e_3, \qquad [e_2, e_3] = e_1.$$
\begin{theorem}
A pseudo-Riemannian inner product on $\mathfrak{g}_{4.8}^{\alpha}$ is an algebraic Ricci soliton if and only if its Ricci operator satisfies the following identities
\begin{multicols}{2}
\begin{itemize}
\item $R_{i1} = 0, \; i = 2, 3, 4$
\item $R_{34} = -\alpha R_{12}$
\item $\alpha R_{32} = R_{32}$
\item $R_{42} = R_{43} = 0$
\item $R_{24} = R_{13}$
\item $\alpha R_{23} = R_{23}$
\item $R_{44} = R_{22} + R_{33} - R_{11}$
\end{itemize}
\end{multicols}
Under the stated conditions, the Ricci operator can be expressed as $\operatorname{Ric} = \eta I_4 + D$ where $\eta = R_{44} = R_{22} + R_{33} - R_{11}$ and
$$D = \begin{bmatrix}
R_{11} - R_{44} & R_{12} & R_{13} & R_{14}\\
\\
0 & R_{22} - R_{44} & R_{23} & R_{13}\\
\\
0 & R_{32} & R_{33} - R_{44} & -\alpha R_{12}\\
\\
0 & 0 & 0 & 0
\end{bmatrix}$$
is a derivation of $\mathfrak{g}_{4.8}^{\alpha}$ with respect to the basis $\mathcal{B}$. Note that if $R_{23} = R_{32} = 0$, then there is no additional condition over $\alpha$. But we see that if $R_{23} \neq 0$ or $R_{32} \neq 0$  then $\alpha = 1$.
\end{theorem}
\begin{proof}
As a derivation of $\mathfrak{g}_{4.8}^{\alpha}$, $D$ in (\ref{derivation}) satisfies the following property
\begin{eqnarray*}
& [D(e_1), e_4] + [e_1, D(e_4)] = D[e_1, e_4] \hspace{6.25cm}\\ \Leftrightarrow & \left[ \left( R_{11} - \eta\right)e_1 + R_{21}e_2 + R_{31}e_3 + R_{41}e_4, e_4\right] + \hspace{4.5cm}\\ & \left[ e_1, R_{14}e_1 + R_{24}e_2 + R_{34}e_3 + \left(R_{44} - \eta\right)e_4\right] = (1 + \alpha)D(e_1)\hspace{2.1cm}\\
\Leftrightarrow& (1 + \alpha)\left( R_{11} - \eta\right)e_1 + R_{21}e_2 + \alpha R_{31}e_3 + (1 + \alpha)\left( R_{44} - \eta\right)e_1\hspace{1.7cm}\\
& = (1 + \alpha)\left( R_{11} - \eta\right)e_1 + (1 + \alpha)R_{21}e_2 + (1 + \alpha)R_{31}e_3 + (1 + \alpha)R_{41}e_4
\end{eqnarray*}
An analysis of the components of $e_i$ shows that
$$\left\lbrace\begin{array}{lll}
R_{11} - \eta + R_{44} - \eta = R_{11} - \eta \Longrightarrow \eta = R_{44} \\
\\
R_{21} = (1 + \alpha)R_{21}\\
\\
\alpha R_{31} = (1 + \alpha)R_{31} \Longrightarrow R_{31} = 0, \quad (1 + \alpha)R_{41} = 0 \Longrightarrow R_{41} = 0
\end{array}\right.$$
Also, $D$ satisfies the following property
\begin{eqnarray*}
& [D(e_2), e_4] + [e_2, D(e_4)] = D[e_2, e_4] \hspace{4.6cm}\\ \Leftrightarrow & \left[ R_{12}e_1 + \left( R_{22} - \eta\right)e_2 + R_{32}e_3 + R_{42}e_4, e_4\right] + \hspace{2.9cm}\\ & \left[ e_2, R_{14}e_1 + R_{24}e_2 + R_{34}e_3 + \left(R_{44} - \eta\right)e_4\right] = D(e_2)\hspace{1.8cm}\\
\Leftrightarrow& (1 + \alpha)R_{12}e_1 + \left( R_{22} - \eta\right)e_2 + \alpha R_{32}e_3 + R_{34}e_1 + \left( R_{44} - \eta\right)e_2\\
& = R_{12}e_1 + \left( R_{22} - \eta\right)e_2 + R_{32}e_3 + R_{42}e_4\hspace{3.7cm}
\end{eqnarray*}
Upon identifying the components of $e_i$, we conclude that
$$\left\lbrace\begin{array}{lll}
(1 + \alpha)R_{12} + R_{34} = R_{12} \Longrightarrow R_{34} = -\alpha R_{12} \\
\\
R_{22} - \eta + R_{44} - \eta = R_{22} - \eta \Longrightarrow \eta = R_{44}\\
\\
\alpha R_{32} = R_{32}, \qquad R_{42} = 0
\end{array}\right.$$
As well, we have
\begin{eqnarray*}
& [D(e_3), e_4] + [e_3, D(e_4)] = D[e_3, e_4] \hspace{4.95cm}\\ \Leftrightarrow & \left[ R_{13}e_1 + R_{23}e_2 + \left( R_{33} - \eta\right)e_3 + R_{43}e_4, e_4\right] + \hspace{3.23cm}\\ & \left[ e_3, R_{14}e_1 + R_{24}e_2 + R_{34}e_3 + \left(R_{44} - \eta\right)e_4\right] = \alpha D(e_3)\hspace{1.9cm}\\
\Leftrightarrow& (1 + \alpha)R_{13}e_1 + R_{23}e_2 + \alpha\left( R_{33} - \eta\right)e_3 - R_{24}e_1 + \alpha\left( R_{44} - \eta\right)e_3\\
& = \alpha R_{13}e_1 + \alpha R_{23}e_2 + \alpha\left( R_{33} - \eta\right)e_3 + \alpha R_{43}e_4\hspace{3cm}
\end{eqnarray*}
From the components of $e_i$, we obtain
$$\left\lbrace\begin{array}{lll}
(1 + \alpha)R_{13} - R_{24} = \alpha R_{13} \Longrightarrow R_{24} = R_{13} \\
\\
R_{23} = \alpha R_{23}\\
\\
\alpha(R_{33} - \eta) + \alpha(R_{44} - \eta) = \alpha(R_{33} - \eta)\\
\\
\alpha R_{43} = 0
\end{array}\right.$$
Proceeding analogously, $D$ satisfies the following property (note that $R_{42} = 0$)
\begin{eqnarray*}
& [D(e_2), e_3] + [e_2, D(e_3)] = D[e_2, e_3] \hspace{4.1cm}\\ \Leftrightarrow & \left[ R_{12}e_1 + \left( R_{22} - \eta\right)e_2 + R_{32}e_3, e_3\right] +\hspace{3.88cm}\\ & \left[ e_2, R_{13}e_1 + R_{23}e_2 + \left( R_{33} - \eta\right)e_3 + R_{43}e_4\right] = D(e_1)\hspace{1.3cm}\\
\Leftrightarrow& \left( R_{22} - \eta\right)e_1 + \left( R_{33} - \eta\right)e_1 + R_{43}e_2 = \left( R_{11} - \eta\right)e_1 + R_{21}e_2
\end{eqnarray*}
By identifying each component of $e_i$, we deduce that
$$\left\lbrace\begin{array}{lll}
R_{22} - \eta + R_{33} - \eta = R_{11} - \eta \Longrightarrow \eta = R_{22} + R_{33} - R_{11} \\
\\
R_{43} = R_{21}
\end{array}\right.$$
Next, the condition $[D(e_1), e_2] + [e_1, D(e_2)] = 0$ is trivial. But we have
\begin{eqnarray*}
& [D(e_1), e_3] + [e_1, D(e_3)] = 0 \hspace{5.88cm}\\ \Leftrightarrow & \left[ R_{21}e_2, e_3\right] + \left[ e_1, R_{13}e_1 + R_{23}e_2 + \left( R_{33} - \eta\right)e_3 + R_{43}e_4\right] = 0\hspace{0.38cm}\\
\Leftrightarrow & R_{21}e_1 + \left( 1 + \alpha\right)R_{43}e_1 = 0 \Leftrightarrow R_{21} + \left( 1 + \alpha\right)R_{43} = 0\hspace{1.7cm}\\
\Leftrightarrow & R_{21} = -(1 + \alpha)R_{21}\; (\text{because}\; R_{43} = R_{21}) \Longrightarrow R_{21} = R_{43} = 0 
\end{eqnarray*}
Thus, we have established the desired result.
\end{proof}
\subsection{The Lie algebra $\mathfrak{g}_{4.9}^{0}$}
The Lie algebra $\mathfrak{g}_{4.9}^{0}$ has a basis $\mathcal{B} = \left\lbrace e_1, e_2, e_3, e_4\right\rbrace$ such that the nonzero Lie brackets are 
$$[e_2, e_3] = e_1, \qquad [e_2, e_4] = -e_3, \qquad [e_3, e_4] = e_2.$$
\begin{theorem}
A pseudo-Riemannian inner product on $\mathfrak{g}_{4.9}^{0}$ is an algebraic Ricci soliton if and only if its Ricci operator satisfies
\begin{multicols}{2}
\begin{itemize}
\item $R_{i1} = 0, \; i = 2, 3, 4$
\item $R_{42} = R_{43} = 0$
\item $R_{34} = -R_{13}$
\item $R_{32} = -R_{23}$
\item $R_{24} = -R_{12}$
\item $R_{44} = R_{22} + R_{33} - R_{11}$
\end{itemize}
\end{multicols}
Provided these conditions are satisfied, the Ricci operator can take the form\\ $\operatorname{Ric} = \eta I_4 + D$ where $\eta = R_{44} = R_{22} + R_{33} - R_{11}$ and
$$D = \begin{bmatrix}
R_{11} - R_{44} & R_{12} & R_{13} & R_{14}\\
\\
0 & R_{22} - R_{44} & R_{23} & -R_{12}\\
\\
0 & -R_{23} & R_{33} - R_{44} & -R_{13}\\
\\
0 & 0 & 0 & 0
\end{bmatrix}$$
is a derivation of $\mathfrak{g}_{4.9}^{0}$ with respect to the basis $\mathcal{B}$.
\end{theorem}
\begin{proof}
Because $D$ in (\ref{derivation}) is a derivation of $\mathfrak{g}_{4.9}^{0}$, we have
\begin{eqnarray*}
& [D(e_2), e_3] + [e_2, D(e_3)] = D[e_2, e_3] \hspace{8.7cm}\\ \Leftrightarrow & \left[ R_{12}e_1 + \left( R_{22} - \eta\right)e_2 + R_{32}e_3 + R_{42}e_4, e_3\right] + \hspace{6.95cm}\\ & \left[ e_2, R_{13}e_1 + R_{23}e_2 + \left(R_{33} - \eta\right)e_3 + R_{43}e_4\right] = D(e_1)\hspace{5.85cm}\\
\Leftrightarrow& \left( R_{22} - \eta\right)e_1 - R_{42}e_2 + \left( R_{33} - \eta\right)e_1 - R_{43}e_3 = \left( R_{11} - \eta\right)e_1 + R_{21}e_2 + R_{31}e_3 + R_{41}e_4
\end{eqnarray*}
By identifying the components of $e_i$, we obtain that
$$\left\lbrace\begin{array}{lll}
R_{22} - \eta + R_{33} - \eta = R_{11} - \eta \Longrightarrow \eta = R_{22} + R_{33} - R_{11} \\
\\
R_{42} = -R_{21}, \qquad R_{43} = -R_{31}, \qquad R_{41} = 0
\end{array}\right.$$
By the same argument, we have
\begin{eqnarray*}
& [D(e_2), e_4] + [e_2, D(e_4)] = D[e_2, e_4] \hspace{3.15cm}\\ \Leftrightarrow & \left[ R_{12}e_1 + \left( R_{22} - \eta\right)e_2 + R_{32}e_3 + R_{42}e_4, e_4\right] + \hspace{1.4cm}\\ & \left[ e_2, R_{14}e_1 + R_{24}e_2 + R_{34}e_3 + \left(R_{44} - \eta\right)e_4\right] = -D(e_3)\\
\Leftrightarrow& -\left( R_{22} - \eta\right)e_3 + R_{32}e_2 + R_{34}e_1 - \left( R_{44} - \eta\right)e_3\hspace{1.1cm}\\
& = -R_{13}e_1 - R_{23}e_2 - \left( R_{33} - \eta\right)e_3 - R_{43}e_4\hspace{1.75cm}
\end{eqnarray*}
By examining the components of $e_i$, we find that
$$\left\lbrace\begin{array}{lll}
R_{34} = -R_{13}, \qquad R_{32} = -R_{23}\\
\\
R_{22} - \eta + R_{44} - \eta = R_{33} - \eta \Longrightarrow \eta = R_{22} + R_{44} - R_{33}\\
\\
R_{43} = 0
\end{array}\right.$$
Proceeding analogously, we have
\begin{eqnarray*}
& [D(e_3), e_4] + [e_3, D(e_4)] = D[e_3, e_4] \hspace{2.83cm}\\ \Leftrightarrow & \left[ R_{13}e_1 + R_{23}e_2 + \left( R_{33} - \eta\right)e_3 + R_{43}e_4, e_4\right] +\hspace{1.1cm}\\ & \left[ e_3, R_{14}e_1 + R_{24}e_2 + R_{34}e_3 + \left(R_{44} - \eta\right)e_4\right] = D(e_2)\\
\Leftrightarrow& -R_{23}e_3 + \left( R_{33} - \eta\right)e_2 - R_{24}e_1 + \left( R_{44} - \eta\right)e_2\hspace{0.9cm}\\
& = R_{12}e_1 + \left( R_{22} - \eta\right)e_2 + R_{32}e_3 + R_{42}e_4\hspace{1.95cm}
\end{eqnarray*}
Identifying the components of $e_i$ yields
$$\left\lbrace\begin{array}{lll}
R_{24} = -R_{12}\\
\\
R_{33} - \eta + R_{44} - \eta = R_{22} - \eta \Longrightarrow \eta = R_{33} + R_{44} - R_{22}\\
\\
R_{32} = -R_{23}, \qquad R_{42} = 0
\end{array}\right.$$
By the two later bracket conditions, we see that
$$\eta = R_{22} + R_{44} - R_{33} = R_{33} + R_{44} - R_{22} \Longrightarrow 2\eta = \eta + \eta = 2R_{44} \Longrightarrow \eta = R_{44}.$$
By the first bracket condition, we have $\eta = R_{22} + R_{33} - R_{11}$. Hence 
$$\eta = R_{44} = R_{22} + R_{33} - R_{11}.$$
We have $R_{42} = -R_{21} = 0$ and $R_{43} = -R_{31} = 0$. We replace all the obtained equalities in the derivation $D$, it becomes
$$D = \begin{bmatrix}
R_{11} - R_{44} & R_{12} & R_{13} & R_{14}\\
\\
0 & R_{22} - R_{44} & R_{23} & -R_{12}\\
\\
0 & -R_{23} & R_{33} - R_{44} & -R_{13}\\
\\
0 & 0 & 0 & 0
\end{bmatrix}.$$
Next, it is easy to see that the conditions $[D(e_1), e_i] + [e_1, D(e_i)] = 0$ for $i = 2, 3, 4$ are trivially satisfied by $D$. We have therefore obtained the conditions stated in the theorem.
\end{proof}
\subsection{The Lie algebra $\mathfrak{g}_{4.9}^{\alpha}$}
The Lie algebra $\mathfrak{g}_{4.9}^{\alpha}$ ($\alpha > 0$) has a basis $\mathcal{B} = \left\lbrace e_1, e_2, e_3, e_4\right\rbrace$ such that the nonzero Lie brackets are
$$[e_1, e_4] = 2\alpha e_1, \qquad [e_2, e_4] = \alpha e_2 - e_3, \qquad [e_3, e_4] = e_2 + \alpha e_3, \qquad [e_2, e_3] = e_1.$$
\begin{theorem}
A pseudo-Riemannian inner product on $\mathfrak{g}_{4.9}^{\alpha}$ is an algebraic Ricci soliton if and only if its Ricci operator satisfies
\begin{multicols}{2}
\begin{itemize}
\item $R_{i1} = 0, \; i = 2, 3, 4$
\item $R_{42} = R_{43} = 0$
\item $R_{34} = -\alpha R_{12} - R_{13}$
\item $R_{32} = -R_{23}$
\item $R_{22} = R_{33}$
\item $R_{24} = -R_{12} + \alpha R_{13}$
\item $R_{44} = R_{22} + R_{33} - R_{11}$
\end{itemize}
\end{multicols}
In this case, the Ricci operator can be written as $\operatorname{Ric} = \eta I_4 + D$ where\\ $\eta = R_{44} = R_{22} + R_{33} - R_{11}$ and 
$$D = \begin{bmatrix}
R_{11} - R_{44} & R_{12} & R_{13} & R_{14}\\
\\
0 & R_{22} - R_{44} & R_{23} & -R_{12} + \alpha R_{13}\\
\\
0 & -R_{23} & R_{22} - R_{44} & -\alpha R_{12} - R_{13}\\
\\
0 & 0 & 0 & 0
\end{bmatrix}$$
is a derivation of $\mathfrak{g}_{4.9}^{\alpha}$ with respect to the basis $\mathcal{B}$.
\end{theorem}
\begin{proof}
From the fact that $D$ in (\ref{derivation}) is a derivation of $\mathfrak{g}_{4.9}^{\alpha}$, the following rule holds
\begin{eqnarray*}
& [D(e_1), e_4] + [e_1, D(e_4)] = D[e_1, e_4] \hspace{5.58cm}\\ \Leftrightarrow & \left[ \left( R_{11} - \eta\right)e_1 + R_{21}e_2 + R_{31}e_3 + R_{41}e_4, e_4\right] + \hspace{3.85cm}\\ & \left[ e_1, R_{14}e_1 + R_{24}e_2 + R_{34}e_3 + \left(R_{44} - \eta\right)e_4\right] = 2\alpha D(e_1)\hspace{2.3cm}\\
\Leftrightarrow& 2\alpha\left( R_{11} - \eta\right)e_1 + R_{21}(\alpha e_2 - e_3) + R_{31}(e_2 + \alpha e_3) + 2\alpha\left( R_{44} - \eta\right)e_1\\
& = 2\alpha\left( R_{11} - \eta\right)e_1 + 2\alpha R_{21}e_2 + 2\alpha R_{31}e_3 + 2\alpha R_{41}e_4 \hspace{2.7cm}
\end{eqnarray*}
From the components of $e_i$, it follows that
$$\left\lbrace\begin{array}{lll}
R_{11} - \eta + R_{44} - \eta = R_{11} - \eta \Longrightarrow \eta = R_{44} \\
\\
\alpha R_{21} + R_{31} = 2\alpha R_{21} \Longrightarrow R_{31} = \alpha R_{21}\\
\\
-R_{21} + \alpha R_{31} = 2\alpha R_{31} \Longrightarrow \alpha R_{31} = -R_{21}, \qquad R_{41} = 0
\end{array}\right.$$
We see that: $\alpha R_{31} = \alpha^2 R_{21} = -R_{21}$, then $R_{21} = R_{31} = 0$.
In a similar manner, $D$ satisfies the following rule
\begin{eqnarray*}
& [D(e_2), e_4] + [e_2, D(e_4)] = D[e_2, e_4] \hspace{9.1cm}\\ \Leftrightarrow & \left[ R_{12}e_1 + \left( R_{22} - \eta\right)e_2 + R_{32}e_3 + R_{42}e_4, e_4\right] + \hspace{7.4cm}\\ & \left[ e_2, R_{14}e_1 + R_{24}e_2 + R_{34}e_3 + \left(R_{44} - \eta\right)e_4\right] = \alpha D(e_2) - D(e_3)\hspace{4.5cm}\\
\Leftrightarrow& 2\alpha R_{12}e_1 + \left( R_{22} - \eta\right)(\alpha e_2 - e_3) + R_{32}(e_2 + \alpha e_3) + R_{34}e_1 + \left( R_{44} - \eta\right)(\alpha e_2 - e_3)\hspace{1.2cm}\\
& = \left( \alpha R_{12} - R_{13}\right)e_1 + \left( \alpha(R_{22} - \eta) - R_{23}\right)e_2 + \left( \alpha R_{32} - (R_{33} - \eta)\right)e_3 + \left( \alpha R_{42} - R_{43}\right)e_4
\end{eqnarray*}
An analysis of the components of $e_i$ shows that
$$\left\lbrace\begin{array}{lll}
2\alpha R_{12} + R_{34} = \alpha R_{12} - R_{13} \Longrightarrow R_{34} = -\alpha R_{12} - R_{13} \\
\\
\alpha\left(R_{22} - \eta \right) + R_{32} + \alpha\left(R_{44} - \eta \right) = \alpha\left(R_{22} - \eta \right) - R_{23} \overset{\eta = R_{44}}{\Longrightarrow} R_{32} = -R_{23}\\
\\
-\left(R_{22} - \eta \right) + \alpha R_{32} - \left(R_{44} - \eta \right) = \alpha R_{32} - \left(R_{33} - \eta \right) \overset{\eta = R_{44}}{\Longrightarrow} R_{22} = R_{33}\\
\\
\alpha R_{42} = R_{43}
\end{array}\right.$$
Likewise, we have
\begin{eqnarray*}
& [D(e_3), e_4] + [e_3, D(e_4)] = D[e_3, e_4] \hspace{9.1cm}\\ \Leftrightarrow & \left[ R_{13}e_1 + R_{23}e_2 + \left( R_{33} - \eta\right)e_3 + R_{43}e_4, e_4\right] + \hspace{7.4cm}\\ & \left[ e_3, R_{14}e_1 + R_{24}e_2 + R_{34}e_3 + \left(R_{44} - \eta\right)e_4\right] = D(e_2) + \alpha D(e_3)\hspace{4.5cm}\\
\Leftrightarrow& 2\alpha R_{13}e_1 + R_{23}(\alpha e_2 - e_3) + \left( R_{33} - \eta\right)(e_2 + \alpha e_3) - R_{24}e_1 + \left( R_{44} - \eta\right)(e_2 + \alpha e_3)\hspace{1.2cm}\\
& = \left( R_{12} + \alpha R_{13}\right)e_1 + \left( R_{22} - \eta + \alpha R_{23}\right)e_2 + \left( R_{32} + \alpha(R_{33} - \eta)\right)e_3 + \left( R_{42} + \alpha R_{43}\right)e_4
\end{eqnarray*}
Upon identifying the components of $e_i$, we conclude that
$$\left\lbrace\begin{array}{lll}
2\alpha R_{13} - R_{24} = R_{12} + \alpha R_{13} \Longrightarrow R_{24} = -R_{12} + \alpha R_{13} \\
\\
\alpha R_{23} + R_{33} - \eta + R_{44} - \eta = R_{22} - \eta + \alpha R_{23} \Longrightarrow \eta = R_{33} + R_{44} - R_{22}\\
\\
-R_{23} + \alpha\left(R_{33} - \eta \right) + \alpha\left(R_{44} - \eta \right) = R_{32} + \alpha\left(R_{33} - \eta \right) \overset{\eta = R_{44}}{\Longrightarrow} R_{32} = -R_{23}\\
\\
R_{42} = -\alpha R_{43}
\end{array}\right.$$
We remark that $\alpha R_{42} = -\alpha^2R_{43} = R_{43}$, which implies that $R_{43} = R_{42} = 0$. Next, we move to the following property:
\begin{eqnarray*}
& [D(e_2), e_3] + [e_2, D(e_3)] = D[e_2, e_3] \hspace{4.4cm}\\ \Leftrightarrow & \left[ R_{12}e_1 + \left( R_{22} - \eta\right)e_2 + R_{32}e_3 + R_{42}e_4, e_3\right] + \hspace{2.65cm}\\ & \left[ e_2, R_{13}e_1 + R_{23}e_2 + \left(R_{33} - \eta\right)e_3 + R_{43}e_4\right] = D(e_1)\hspace{1.6cm}\\
\Leftrightarrow& \left( R_{22} - \eta\right)e_1 - R_{42}(e_2 + \alpha e_3) + \left( R_{33} - \eta\right)e_1 + R_{43}(\alpha e_2 - e_3)\\
& = \left( R_{11} - \eta\right)e_1 + R_{21}e_2 + R_{31}e_3 + R_{41}e_4\hspace{3.5cm}
\end{eqnarray*}
From the components of $e_i$, we obtain
$$\left\lbrace\begin{array}{lll}
R_{22} - \eta + R_{33} - \eta = R_{11} - \eta \Longrightarrow \eta = R_{22} + R_{33} - R_{11} \\
\\
-R_{42} + \alpha R_{43} = R_{21} \Longrightarrow \alpha R_{43} = R_{42} + R_{21}\\
\\
-\alpha R_{42} - R_{43} = R_{31} \Longrightarrow R_{43} = -\alpha R_{42} - R_{31}\\
\\
R_{41} = 0
\end{array}\right.$$
But we have $R_{21} = R_{31} = R_{42} = R_{43} = 0$. According to the previous observations, the derivation $D$ becomes 
$$D = \begin{bmatrix}
R_{11} - R_{44} & R_{12} & R_{13} & R_{14}\\
\\
0 & R_{22} - R_{44} & R_{23} & -R_{12} + \alpha R_{13}\\
\\
0 & -R_{23} & R_{22} - R_{44} & -\alpha R_{12} - R_{13}\\
\\
0 & 0 & 0 & 0
\end{bmatrix}.$$
Next, it is easy to see that the conditions $[D(e_1), e_i] + [e_1, D(e_i)] = 0$ for $i = 2, 3$ are trivially satisfied by $D$. Therefore, the conclusion of the theorem follows
\end{proof}
\subsection{The Lie algebra $\mathfrak{g}_{4.10}$}
The Lie algebra $\mathfrak{g}_{4.10}$ has a basis $\mathcal{B} = \left\lbrace e_1, e_2, e_3, e_4\right\rbrace$ such that the nonzero Lie brackets are
$$[e_1, e_3] = e_1, \qquad [e_2, e_3] = e_2, \qquad [e_1, e_4] = -e_2, \qquad [e_2, e_4] = e_1.$$
\begin{theorem}
A pseudo-Riemannian inner product on $\mathfrak{g}_{4.10}$ is an algebraic Ricci soliton if and only if its associated Ricci operator satisfies
\begin{multicols}{2}
\begin{itemize}
\item $R_{31} = R_{41} = R_{32} = 0$
\item $R_{42} = R_{43} = R_{34} = 0$
\item $R_{33} = R_{44}$
\item $R_{22} = R_{11}$
\item $R_{21} = -R_{12}$
\item $R_{14} = R_{23}$
\item $R_{24} = -R_{13}$
\end{itemize}
\end{multicols}
Given these conditions, the Ricci operator can be described as $\operatorname{Ric} = \eta I_4 + D$ where $\eta = R_{44} = R_{33}$ and
$$D = \begin{bmatrix}
R_{11} - R_{33} & R_{12} & R_{13} & R_{23}\\
\\
-R_{12} & R_{11} - R_{33} & R_{23} & -R_{13}\\
\\
0 & 0 & 0 & 0\\
\\
0 & 0 & 0 & 0
\end{bmatrix}$$
is a derivation of $\mathfrak{g}_{4.10}$ with respect to the basis $\mathcal{B}$.
\end{theorem}
\begin{proof}
Given that $D$ in (\ref{derivation}) is a derivation of $\mathfrak{g}_{4.10}$, it satisfies the following condition
\begin{eqnarray*}
& [D(e_1), e_3] + [e_1, D(e_3)] = D[e_1, e_3] \hspace{2.8cm}\\ \Leftrightarrow & \left[ \left( R_{11} - \eta\right)e_1 + R_{21}e_2 + R_{31}e_3 + R_{41}e_4, e_3\right] + \hspace{1.1cm}\\ & \left[ e_1, R_{13}e_1 + R_{23}e_2 + \left(R_{33} - \eta\right)e_3 + R_{43}e_4\right] = D(e_1)\\
\Leftrightarrow& \left( R_{11} - \eta\right)e_1 + R_{21}e_2 + \left( R_{33} - \eta\right)e_1 - R_{43}e_2\hspace{1.2cm}\\
& = \left( R_{11} - \eta\right)e_1 + R_{21}e_2 + R_{31}e_3 + R_{41}e_4\hspace{1.8cm}
\end{eqnarray*}
By identifying each component of $e_i$, we deduce that
$$\left\lbrace\begin{array}{lll}
R_{11} - \eta + R_{33} - \eta = R_{11} - \eta \Longrightarrow \eta = R_{33} \\
\\
R_{21} - R_{43} = R_{21} \Longrightarrow R_{43} = 0\\
\\
R_{31} = R_{41} = 0
\end{array}\right.$$
Similarly, $D$ satisfies the following condition
\begin{eqnarray*}
& [D(e_1), e_4] + [e_1, D(e_4)] = D[e_1, e_4] \hspace{8.02cm}\\ \Leftrightarrow & \left[ \left( R_{11} - \eta\right)e_1 + R_{21}e_2, e_4\right] + \left[ e_1, R_{14}e_1 + R_{24}e_2 + R_{34}e_3 + \left(R_{44} - \eta\right)e_4\right] = - D(e_2)\\
\Leftrightarrow& -\left( R_{11} - \eta\right)e_2 + R_{21}e_1 + R_{34}e_1 - \left( R_{44} - \eta\right)e_2\hspace{6cm}\\
& = -R_{12}e_1 - \left( R_{22} - \eta\right)e_2 - R_{32}e_3 - R_{42}e_4 \hspace{6.8cm}
\end{eqnarray*}
By identifying the components of $e_i$, we obtain that
$$\left\lbrace\begin{array}{lll}
R_{21} + R_{34} = -R_{12} \Longrightarrow R_{34} = -(R_{12} + R_{21})\\
\\
R_{11} - \eta + R_{44} - \eta = R_{22} - \eta \Longrightarrow \eta = R_{11} + R_{44} - R_{22} \\
\\
R_{32} = R_{42} = 0
\end{array}\right.$$
Likewise, we have
\begin{eqnarray*}
& [D(e_2), e_3] + [e_2, D(e_3)] = D[e_2, e_3] \hspace{7.7cm}\\ \Leftrightarrow & \left[ R_{12}e_1 + \left( R_{22} - \eta\right)e_2, e_3\right] + \left[ e_2, R_{13}e_1 + R_{23}e_2 + \left( R_{33} - \eta\right)e_3 + R_{43}e_4\right] = D(e_2)\\
\Leftrightarrow& R_{12}e_1 + \left( R_{22} - \eta\right)e_2 + \left( R_{33} - \eta\right)e_2 + R_{43}e_1 = R_{12}e_1 + \left( R_{22} - \eta\right)e_2\hspace{1.9cm}
\end{eqnarray*}
By examining the components of $e_i$, we find that
$$\left\lbrace\begin{array}{lll}
R_{12} + R_{43} = R_{12} \Longrightarrow R_{43} = 0\\
\\
R_{22} - \eta + R_{33} - \eta = R_{22} - \eta \Longrightarrow \eta = R_{33}
\end{array}\right.$$
Proceeding analogously, we have
\begin{eqnarray*}
& [D(e_2), e_4] + [e_2, D(e_4)] = D[e_2, e_4] \hspace{7.7cm}\\ \Leftrightarrow & \left[ R_{12}e_1 + \left( R_{22} - \eta\right)e_2, e_4\right] + \left[ e_2, R_{14}e_1 + R_{24}e_2 + R_{34}e_3 + \left( R_{44} - \eta\right)e_4\right] = D(e_1)\\
\Leftrightarrow& -R_{12}e_2 + \left( R_{22} - \eta\right)e_1 + R_{34}e_2 + \left( R_{44} - \eta\right)e_1 = \left( R_{11} - \eta\right)e_1 + R_{21}e_2\hspace{1.5cm}
\end{eqnarray*}
Identifying the components of $e_i$ yields
$$\left\lbrace\begin{array}{lll}
R_{22} - \eta + R_{44} - \eta = R_{11} - \eta \Longrightarrow \eta = R_{22} + R_{44} - R_{11}\\
\\
-R_{12} + R_{34} = R_{21} \Longrightarrow R_{34} = R_{12} + R_{21}
\end{array}\right.$$
We remark that $R_{34} = -(R_{12} + R_{21}) = R_{12} + R_{21}$, thus $R_{34} = 0$ and $R_{21} = -R_{12}$. We have $\eta = R_{11} + R_{44} - R_{22} = R_{22} + R_{44} - R_{11}$, hence $2\eta = \eta + \eta = 2R_{44}$. Thus $\eta = R_{44}$, which implies that $R_{11} = R_{22}$. Hence $\eta = R_{44} = R_{33}$ and $R_{11} = R_{22}$. In view of the above facts, the derivation $D$ becomes
$$D = \begin{bmatrix}
R_{11} - R_{33} & R_{12} & R_{13} & R_{14}\\
\\
-R_{12} & R_{11} - R_{33} & R_{23} & R_{24}\\
\\
0 & 0 & 0 & 0\\
\\
0 & 0 & 0 & 0
\end{bmatrix}.$$
Next, we can easily see that the condition $[D(e_1), e_2] + [e_1, D(e_2)] = 0$ is trivial. But we have
\begin{eqnarray*}
& [D(e_3), e_4] + [e_3, D(e_4)] = D[e_3, e_4] = 0\hspace{0.9cm}\\
\Leftrightarrow& \left[ R_{13}e_1 + R_{23}e_2, e_4\right] + \left[ e_3, R_{14}e_1 + R_{24}e_2\right] = 0\\
\Leftrightarrow& -R_{13}e_2 + R_{23}e_1 - R_{14}e_1 - R_{24}e_2 = 0\hspace{1cm}
\end{eqnarray*}
By identifying the components of $e_i$, we obtain that: $R_{14} = R_{23}$ and $R_{24} = -R_{13}$.\\ Thus, the derivation $D$ takes the form
$$D = \begin{bmatrix}
R_{11} - R_{33} & R_{12} & R_{13} & R_{23}\\
\\
-R_{12} & R_{11} - R_{33} & R_{23} & -R_{13}\\
\\
0 & 0 & 0 & 0\\
\\
0 & 0 & 0 & 0
\end{bmatrix}.$$
Therefore, the results of the theorem have been established.
\end{proof}
\begin{example}
We consider the Lie algebra $\mathfrak{g}_{3.1} \oplus \mathfrak{g}_1$, which has a basis $\mathcal{B} = \left\lbrace e_1, e_2, e_3, e_4\right\rbrace$ such that the only nonzero Lie bracket is $[e_2, e_3] = e_1$. Let $\langle \cdot, \cdot\rangle$ be the pseudo-Riemannian inner product on $\mathfrak{g}_{3.1} \oplus \mathfrak{g}_1$ represented in the basis $\mathcal{B}$ by the following matrix
$$\langle \cdot, \cdot\rangle = \begin{bmatrix}
1 & 0 & 0 & 0\\
0 & 0 & 1 & 0\\
0 & 1 & 0 & 0\\
0 & 0 & 0 & -1
\end{bmatrix}.$$
Then, the latter is of signature $(2, 2)$. An orthonormal basis $\mathscr{B} = \left\lbrace v_1, v_2, v_3, v_4\right\rbrace$ of $\left(\mathfrak{g}_{3.1} \oplus \mathfrak{g}_1, \langle \cdot, \cdot\rangle\right)$ is given by
$$v_1 = e_1, \qquad v_2 = \frac{1}{\sqrt{2}}(e_2 + e_3), \qquad v_3 = \frac{1}{\sqrt{2}}(e_2 - e_3), \qquad v_4 = e_4.$$
The only nonzero Lie bracket in the basis $\mathscr{B}$ is $[v_2, v_3] = -v_1$. Thus, the only nonzero structure constants of $\left(\mathfrak{g}_{3.1} \oplus \mathfrak{g}_1, \langle \cdot, \cdot\rangle\right)$ are: $\xi_{231} = -1$ and $\xi_{321} = 1$. Now, by using the formula (\ref{Levi}), we obtain that the Levi-Civita connection of $\left(\mathfrak{g}_{3.1} \oplus \mathfrak{g}_1, \langle \cdot, \cdot\rangle\right)$ is described by the following formulas:
\begin{multicols}{2}
\begin{itemize}
\item $\nabla_{v_i} v_i = 0 \,\,\forall i= 1,\ldots, 4$
\item $\nabla_{v_1} v_2 = \nabla_{v_2} v_1 = \frac{-1}{2}v_3$
\item $\nabla_{v_1} v_3 = \nabla_{v_3} v_1 = \frac{-1}{2}v_2$
\item $\nabla_{v_1} v_4 = \nabla_{v_4} v_1 = 0$
\item $\nabla_{v_2} v_3 = -\nabla_{v_3} v_2 = \frac{-1}{2}v_1$
\item $\nabla_{v_2} v_4 = \nabla_{v_4} v_2 = 0$
\item $\nabla_{v_3} v_4 = \nabla_{v_4} v_3 = 0$
\end{itemize}
\end{multicols}
Next, by using the formula (\ref{Ricci}), we obtain that the Ricci operator of $\left(\mathfrak{g}_{3.1} \oplus \mathfrak{g}_1, \langle \cdot, \cdot\rangle\right)$ is determined by the following identities:
\begin{eqnarray*}
\operatorname{Ric}(v_1) &=& R_{v_2v_1}v_2 - R_{v_3v_1}v_3 - R_{v_4v_1}v_4 = \frac{-1}{2} v_1\\
\operatorname{Ric}(v_2) &=& R_{v_1v_2}v_1 - R_{v_3v_2}v_3 - R_{v_4v_2}v_4 = \frac{1}{2} v_2\\
\operatorname{Ric}(v_3) &=& R_{v_1v_3}v_1 + R_{v_2v_3}v_2 - R_{v_4v_3}v_4 = \frac{1}{2} v_3\\
\operatorname{Ric}(v_4) &=& R_{v_1v_4}v_1 + R_{v_2v_4}v_2 - R_{v_3v_4}v_3 = 0. 
\end{eqnarray*}
Therefore, the matrix of the Ricci operator of $\left(\mathfrak{g}_{3.1} \oplus \mathfrak{g}_1, \langle \cdot, \cdot\rangle\right)$ is given in the basis $\mathscr{B}$ by
$$[\operatorname{Ric}]_{\left\lbrace v_i\right\rbrace} = \operatorname{diag}\left\lbrace \frac{-1}{2}, \frac{1}{2}, \frac{1}{2}, 0\right\rbrace.$$
The transition matrix $A = \left[ \operatorname{id}\right]_{\left\lbrace e_i\right\rbrace, \left\lbrace v_i\right\rbrace}$ is given by
$$A = \begin{bmatrix}
1 & 0 & 0 & 0\\
0 & \frac{1}{\sqrt{2}} & \frac{1}{\sqrt{2}} & 0\\
0 & \frac{1}{\sqrt{2}} & \frac{-1}{\sqrt{2}} & 0\\
0 & 0 & 0 & 1
\end{bmatrix}.$$
We can easily see that $[\operatorname{Ric}]_{\left\lbrace v_i\right\rbrace} \times A = A \times [\operatorname{Ric}]_{\left\lbrace v_i\right\rbrace}$. Thus, the matrix of the Ricci operator of $\left(\mathfrak{g}_{3.1} \oplus \mathfrak{g}_1, \langle \cdot, \cdot\rangle\right)$ is given in the basis $\mathcal{B}$ by
$$[\operatorname{Ric}]_{\left\lbrace e_i\right\rbrace} = A^{-1} \times [\operatorname{Ric}]_{\left\lbrace v_i\right\rbrace} \times A = [\operatorname{Ric}]_{\left\lbrace v_i\right\rbrace}.$$
Furthermore, we can see that the entries $R_{ij}$ of the matrix $[\operatorname{Ric}]_{\left\lbrace e_i\right\rbrace} = [\operatorname{Ric}]_{\left\lbrace v_i\right\rbrace}$ verify the conditions of Theorem \ref{theorem3.3}. Hence the pseudo-Riemannian inner product $\langle \cdot, \cdot\rangle$ is an algebraic Ricci soliton where the constant $\eta$ is given in Theorem \ref{theorem3.3} by
$$\eta = R_{22} + R_{33} - R_{11} = \frac{1}{2} + \frac{1}{2} + \frac{1}{2} = \frac{3}{2},$$
and the derivation $D$ is given in Theorem \ref{theorem3.3} by
$$D = \begin{bmatrix}
R_{11} - \eta & R_{12} & R_{13} & R_{14}\\
\\
0 & R_{22} - \eta & R_{23} & 0\\
\\
0 & R_{32} & R_{33} - \eta & 0\\
\\
0 & R_{42} & R_{43} & R_{44} - \eta
\end{bmatrix} = \begin{bmatrix}
-2 & 0 & 0 & 0\\
\\
0 & -1 & 0 & 0\\
\\
0 & 0 & -1 & 0\\
\\
0 & 0 & 0 & \frac{-3}{2}
\end{bmatrix}.$$
Consequently: $[\operatorname{Ric}]_{\left\lbrace e_i\right\rbrace} = \operatorname{diag}\left\lbrace \frac{-1}{2}, \frac{1}{2}, \frac{1}{2}, 0\right\rbrace = \frac{3}{2}I_{4} + D$, where $D$ as above is a derivation of $\mathfrak{g}_{3.1} \oplus \mathfrak{g}_1$ with respect to the basis $\mathcal{B} = \left\lbrace e_1, e_2, e_3, e_4\right\rbrace$.
\end{example}
An important class of pseudo-Riemannian metrics on Lie groups, or equivalently pseudo-Riemannian inner products on Lie algebras, is the class of flat metrics. The study of these pseudo-Riemannian metrics on Lie groups is a relatively new and intricate area of research, see for example the following works \cite{boucetta2019flat,lebzioui2020flat,benayadi2022flat,lebzioui2023flat}. Such metrics are trivial algebraic Ricci solitons and trivial Einstein metrics. In the following example, we consider a pseudo-Riemannian inner product of signature $(2,2)$ on the Lie algebra $\mathfrak{g}_{3.1} \oplus \mathfrak{g}_1$ and show that it is flat.
\begin{example}
In the Lie algebra $\mathfrak{g}_{3.1} \oplus \mathfrak{g}_1$ with basis $\mathcal{B} = \left\lbrace e_1, e_2, e_3, e_4\right\rbrace$ such that $[e_2, e_3] = e_1$, let us consider $\langle \cdot, \cdot\rangle$ to be the pseudo-Riemannian inner product represented in the basis $\mathcal{B}$ by the following matrix
$$\langle \cdot, \cdot\rangle = \begin{bmatrix}
0 & 1 & 0 & 0\\
1 & 0 & 0 & 0\\
0 & 0 & 0 & 1\\
0 & 0 & 1 & 0
\end{bmatrix}.$$
Then, the latter is of signature $(2, 2)$. An orthonormal basis $\mathscr{B} = \left\lbrace v_1, v_2, v_3, v_4\right\rbrace$ of $\left(\mathfrak{g}_{3.1} \oplus \mathfrak{g}_1, \langle \cdot, \cdot\rangle\right)$ is given by
$$v_1 = \frac{1}{\sqrt{2}}(e_1 + e_2), \qquad v_2 = \frac{1}{\sqrt{2}}(e_3 + e_4), \qquad v_3 = \frac{1}{\sqrt{2}}(e_1 - e_2), \qquad v_4 = \frac{1}{\sqrt{2}}(e_3 - e_4).$$
The nonzero commutators in the basis $\mathscr{B}$ are given by the following relations
\begin{multicols}{2}
\begin{itemize}
\item $[v_1, v_2] = \frac{1}{2\sqrt{2}}v_1 + \frac{1}{2\sqrt{2}}v_3$
\item $[v_1, v_4] = \frac{1}{2\sqrt{2}}v_1 + \frac{1}{2\sqrt{2}}v_3$
\item $[v_2, v_3] = \frac{1}{2\sqrt{2}}v_1 + \frac{1}{2\sqrt{2}}v_3$
\item $[v_3, v_4] = \frac{-1}{2\sqrt{2}}v_1 - \frac{1}{2\sqrt{2}}v_3$
\end{itemize}
\end{multicols}
Thus, the nonzero structure constants of $\left(\mathfrak{g}_{3.1} \oplus \mathfrak{g}_1, \langle \cdot, \cdot\rangle\right)$ are:
\begin{multicols}{4}
\begin{itemize}
\item $\xi_{121} = \frac{1}{2\sqrt{2}}$
\item $\xi_{211} = \frac{-1}{2\sqrt{2}}$
\item $\xi_{123} = \frac{-1}{2\sqrt{2}}$
\item $\xi_{213} = \frac{1}{2\sqrt{2}}$
\item $\xi_{141} = \frac{1}{2\sqrt{2}}$
\item $\xi_{411} = \frac{-1}{2\sqrt{2}}$
\item $\xi_{143} = \frac{-1}{2\sqrt{2}}$
\item $\xi_{413} = \frac{1}{2\sqrt{2}}$
\item $\xi_{231} = \frac{1}{2\sqrt{2}}$
\item $\xi_{321} = \frac{-1}{2\sqrt{2}}$
\item $\xi_{233} = \frac{-1}{2\sqrt{2}}$
\item $\xi_{323} = \frac{1}{2\sqrt{2}}$
\item $\xi_{341} = \frac{-1}{2\sqrt{2}}$
\item $\xi_{431} = \frac{1}{2\sqrt{2}}$
\item $\xi_{343} = \frac{1}{2\sqrt{2}}$
\item $\xi_{433} = \frac{-1}{2\sqrt{2}}$
\end{itemize}
\end{multicols}
Now, by using the formula (\ref{Levi}), we obtain that the Levi-Civita connection of $\left(\mathfrak{g}_{3.1} \oplus \mathfrak{g}_1, \langle \cdot, \cdot\rangle\right)$ is described by the following identities:
\begin{multicols}{2}
\begin{itemize}
\item $\nabla_{v_1} v_1 = \frac{-1}{2\sqrt{2}}v_2 + \frac{1}{2\sqrt{2}}v_4$
\item $\nabla_{v_3} v_3 = \frac{-1}{2\sqrt{2}}v_2 + \frac{1}{2\sqrt{2}}v_4$
\item $\nabla_{v_2} v_2 = \nabla_{v_4} v_4 = 0$
\item $\nabla_{v_1} v_2 = \frac{1}{2\sqrt{2}}v_1 + \frac{1}{2\sqrt{2}}v_3$
\item $\nabla_{v_2} v_1 = 0$
\item $\nabla_{v_1} v_3 = \frac{1}{2\sqrt{2}}v_2 - \frac{1}{2\sqrt{2}}v_4$
\item $\nabla_{v_3} v_1 = \frac{1}{2\sqrt{2}}v_2 - \frac{1}{2\sqrt{2}}v_4$
\item $\nabla_{v_1} v_4 = \frac{1}{2\sqrt{2}}v_1 + \frac{1}{2\sqrt{2}}v_3$
\item $\nabla_{v_4} v_1 = \nabla_{v_2} v_3 = 0$
\item $\nabla_{v_3} v_2 = \frac{-1}{2\sqrt{2}}v_1 - \frac{1}{2\sqrt{2}}v_3$
\item $\nabla_{v_2} v_4 = \nabla_{v_4} v_2 = 0$
\item $\nabla_{v_3} v_4 = \frac{-1}{2\sqrt{2}}v_1 - \frac{1}{2\sqrt{2}}v_3$
\item $\nabla_{v_4} v_3 = 0$
\end{itemize}
\end{multicols}
After very careful and exacting calculations, we obtain the following result
$$R_{v_1v_2} = R_{v_1v_3} = R_{v_1v_4} = R_{v_2v_3} = R_{v_2v_4} = R_{v_3v_4} = 0.$$
For example, we have
\begin{eqnarray*}
R_{v_1v_2} v_1 &=& \nabla_{[v_1, v_2]} v_1 - \nabla_{v_1}\nabla_{v_2} v_1 + \nabla_{v_2}\nabla_{v_1} v_1\\
&=& \frac{1}{2\sqrt{2}}\nabla_{v_1} v_1 + \frac{1}{2\sqrt{2}}\nabla_{v_3} v_1 - \frac{1}{2\sqrt{2}}\nabla_{v_2} v_2 + \frac{1}{2\sqrt{2}}\nabla_{v_2} v_4\\
&=& 0.
\end{eqnarray*}
Similarly, one sees that 
$$R_{v_1v_2} v_2 = R_{v_1v_2} v_3 = R_{v_1v_2} v_4 = 0.$$
Thus $R_{v_1v_2} = 0$, and the same strategy applies to the other transformations. Next, by skew-symmetry, $R_{v_iv_j} = - R_{v_jv_i}$, so all the remaining transformations vanish. Consequently, the Riemann tensor $R$ vanishes, and the pseudo-Riemannian inner product $\langle \cdot, \cdot\rangle$ is flat.
\end{example}

\end{document}